\documentclass[hidelinks,onefignum,onetabnum]{siamart220329}
\usepackage{amssymb,mathtools}

\newsiamthm{assumption}{Assumption}
\newsiamremark{remark}{Remark}

\newcommand{\R}{\mathbb R}
\newcommand{\Id}{\operatorname{id}}
\newcommand{\dd}{\,\mathrm d}
\newcommand{\norm}[2]{\lVert#1\rVert_{#2}}
\newcommand{\gr}{\nabla_\Gamma}
\newcommand{\lap}{\Delta_\Gamma}
\newcommand{\dv}{\nabla_\Gamma\cdot}
\newcommand{\tr}{\operatorname{tr}}
\newcommand{\dman}{d_{\mathcal{M}}}
\allowdisplaybreaks[2]
\headers{Convergence of isoparametric BGN for surfaces}{Y. Li}

\title{Convergence of isoparametric BGN methods
for mean curvature flow of closed surfaces with high-order finite elements}
\author{Yifei Li\thanks{Mathematisches Institut, Universit\"at T\"ubingen, Auf der Morgenstelle 10, 72076 T\"ubingen, Germany (\email{yifei.li@mnf.uni-tuebingen.de}).}}
\begin{document}
\maketitle

\begin{abstract}
We prove the convergence of the isoparametric Barrett--Garcke--N\"urnberg (BGN) method for mean curvature flow of closed surfaces. For finite elements of degree $k\ge6$ and time steps $\tau=O(h^2)$, the numerical surfaces converge to the exact surfaces in the manifold distance with order $h^2$. The proof relies on two ingredients. First, we construct approximation flows by an expansion in time, whose interpolations satisfy the BGN method with an improved defect. Second, we derive a refined estimate for the tangential error from an algebraic identity for tangential test functions.
\end{abstract}



\section{Introduction}\label{sec:intro}

Parametric finite element methods were first proposed by Dziuk \cite{Dziuk90} for mean curvature flow, and have since been widely used for the numerical simulation of various geometric flows, including surface diffusion and Willmore flow. Barrett, Garcke and N\"urnberg (BGN) \cite{BGN07,BGN08} introduced a different variational formulation that allows the tangential motion, which keeps the mesh in good quality during the evolution without remeshing. The BGN methods have shown desirable performance in computations and have been applied to a wide range of problems, including two-phase flows and biomembranes \cite{BGN13,BGN15,BGN16}, crystal growth and solid-state dewetting \cite{BGN10,BJWZ17,ZJB20,BZ21,GNZ26}, and anisotropic geometric flows \cite{BGN08A,BJL23,BL24}; see the review paper \cite{BGN20}. Numerical experiments show that the BGN methods converge in the manifold distance introduced in \cite{ZJB21}. However, even for mean curvature flow, the convergence of the BGN method for closed surfaces has not been proved.

The techniques developed for evolving surface finite element methods \cite{DE07} provide a powerful tool for the error analysis of parametric methods. They were first developed for surfaces with prescribed velocity \cite{DE07,KLLP,Kovacs}, and then extended to geometric flows whose velocity is determined by PDEs on the surface \cite{KLL19,HuLi,LiSINUM,BaiLiErratum,BaiLi,HLT}, including mean curvature flow, surface diffusion and Willmore flow. In these works, the velocity is purely determined by PDEs, and the spatial semi-discretization is well-defined. In contrast, the spatial semi-discretization of the BGN method is a differential-algebraic equation instead of a PDE, so that the semi-discrete tangential velocity is infinite, whereas the full discretization is uniquely solvable \cite{BGN07,BGN08}. Hence these analyses cannot be directly applied to the BGN method. The existing convergence results for BGN-type methods of curves \cite{BaiLiMC,BGV} add stabilization terms to the BGN method so that the tangential velocity is uniquely determined, while the convergence of the original BGN method without stabilization has remained open since \cite{BGN07,BGN08}.

Recently, the convergence of the original BGN method for curves was proved in \cite{Li} by a space-time backward error analysis, in which approximation flows are introduced to obtain an improved defect. However, there are two essential difficulties in extending this approach to closed surfaces. First, the construction of the spatial part of the approximation flows depends on the structure of the mesh. For curves, all elements of a uniform mesh are identical. This is in general not the case for closed surfaces. Suppose that every vertex has the same number $q$ of adjacent vertices, and let $V$, $E$ and $F$ be the numbers of vertices, edges and triangles, respectively, so that $qV=2E=3F$. Then Euler's formula $V-E+F=2-2g$ gives $(6-q)V=12(1-g)$, where $g$ is the genus of the surface. Hence, unless $g=1$, such structured meshes cannot be refined arbitrarily, and for a closed surface which is not a topological torus, the spatial expansion of \cite{Li} cannot even be carried out. Second, the tangent space of a curve is one-dimensional, so that every tangential test function is the interpolation of a scalar function times a discrete tangent vector, and the tangential error can be computed explicitly \cite[Lemma~3.6]{Li}. For surfaces, the tangent space is two-dimensional and such an explicit computation is not available. The tangential error then contains a term involving the normal error, which cannot be estimated directly with the required order.

In this paper, we prove that the isoparametric BGN method \cite{GNPZ} for mean curvature flow of closed surfaces with finite elements of degree $k\ge6$ converges in the manifold distance with an error of $O(h^2)$ under the time step condition $\tau=O(h^2)$. To overcome the first difficulty, we observe that the dominant terms in the space-time expansion are the temporal ones, for finite elements of sufficiently high degree. Therefore, the method behaves essentially as a semi-discretization in time, and it suffices to expand in time. The resulting approximation flows satisfy the BGN method with a normal defect of $O(\tau^2+h^{k-1})$ and a tangential defect of $O(\tau^3+\tau^2h^2+h^k)$, which is amplified by $\tau^{-1}$ in the error analysis. For $k\ge6$ and $\tau=O(h^2)$, both the normal defect and $\tau^{-1}$ times the tangential defect are of order $h^4$.

For the second difficulty, we start from the geometric identity $\int_\Gamma\nabla_\Gamma\Id\cdot\nabla_\Gamma\phi=\int_\Gamma Hn\cdot\phi$, which holds for every test function $\phi$ on every smooth surface, where $n$ is the unit normal vector. Differentiating both sides along a perturbation of the surface gives two different expressions of the same derivative, and hence an algebraic identity. For tangential test functions $(I-n\otimes n)\phi$, this identity helps us control the tangential error. In its discrete version, superapproximation yields a factor $h^2$ for tangential test functions, which has to be balanced by the factor $\tau$ in front of the tangential part of the BGN method. This explains why $\tau=O(h^2)$ is the natural time step condition; see Remark~\ref{rem:time-step}. The discrete version contains a normal error term $D_h$, which we control precisely by inserting the decomposition of the stiffness term by Bai and Li \cite{BaiLi} into the error equation; see Lemma~\ref{lem:D-bound}.

The rest of this paper is organized as follows. In Section~\ref{sec:discretization}, we describe the isoparametric finite element discretization and the BGN method, and state the main result. In Section~\ref{sec:approximation-flows}, we construct the approximation flows by an expansion in time and prove the improved defect estimates. Section~\ref{sec:proof} is devoted to the error analysis and the proof of Theorem~\ref{thm:main}. The proof of Lemma~\ref{lem:tangential-identity} is given in Appendix~\ref{app:tangential-identity}.

\section{Finite element discretization and the main result}\label{sec:discretization}

\subsection{Initial approximation and finite element space}\label{subsec:initial-triangulation}

Let $\Gamma^0\subset\R^3$ be a connected, closed, oriented smooth surface, let $D_\delta(\Gamma^0) = \{x\in\R^3:\operatorname{dist}(x,\Gamma^0)\le\delta\}$ be a sufficiently small tubular neighborhood of $\Gamma^0$, and let $a^0:D_\delta(\Gamma^0)\to\Gamma^0$ denote the distance projection onto $\Gamma^0$.

Following the standard isoparametric construction \cite[Section~2]{Demlow}, let $\Gamma_h^0$ be a degree-$k$ isoparametric interpolation of $\Gamma^0$ associated with a shape-regular and quasi-uniform triangulation $\mathcal T_h^0$. We set
\begin{equation*}
	h=\max_{K^0\in\mathcal T_h^0}\operatorname{diam}K^0, \qquad ch\le \operatorname{diam}K^0\le h \qquad\text{for all }K^0\in\mathcal T_h^0,
\end{equation*}
where $c>0$ is independent of $h$. For sufficiently small $h$, the restriction $a^0|_{\Gamma_h^0}$ is a bijection.

For each $K^0\in\mathcal T_h^0$, let $F_{K^0}:\widehat K\to K^0$ denote the standard degree-$k$ isoparametric element map from the reference triangle $\widehat K$. The finite element space on $\Gamma_h^0$ is
\begin{equation*}
	S_h^0 = \left\{ u_h\in C(\Gamma_h^0): u_h\circ F_{K^0}\in\mathbb P_k(\widehat K) \quad\text{for all }K^0\in\mathcal T_h^0 \right\}.
\end{equation*}
We denote by $p_1^0,\ldots,p_J^0$ the Lagrange nodes of $\Gamma_h^0$. For $f\in C(\Gamma^0)$, the nodal interpolation operator $I_h$ is defined by
\begin{equation*}
	(I_hf)(p_j^0)=f(a^0(p_j^0)), \qquad j=1,\ldots,J.
\end{equation*}

The standard geometric approximation properties of $\Gamma_h^0$ and $\Gamma^0$ will be used throughout without further comment; see \cite[Propositions~2.3 and~2.7]{Demlow}.

\subsection{Exact and interpolated evolving surfaces}\label{subsec:evolving-surfaces}

Let $X:\Gamma^0\times[0,T]\to\R^3$ be a regular parametrization satisfying $X(\cdot,0)=\Id_{\Gamma^0}$ and $X(\cdot,t)\in W^{k+2,\infty}(\Gamma^0)^3$, $|\nabla_{\Gamma^0} X(t)| \geq c>0$ for all $t\in[0,T]$. The exact evolving surface is given by $\Gamma(t)=X(\Gamma^0,t)$. The interpolated parametrization and the corresponding interpolated surface are $X_{h,*}(t)=I_hX(t)$ and $\Gamma_{h,*}(t)=X_{h,*}(t)(\Gamma_h^0)$. The finite element space on $\Gamma_{h,*}(t)$ is denoted by
\begin{equation*}
	S_{h,*}(t) = S_h(\Gamma_{h,*}(t)) = \{u_h\circ (X_{h,*}(t))^{-1}:u_h\in S_h^0\}.
\end{equation*}
$I_h:=I_{h,\Gamma_{h,*}(t)}$ also denotes nodal interpolation on $\Gamma_{h,*}(t)$. When no confusion can arise, the time dependence is omitted.

There exists $\delta>0$ such that, for all $t\in[0,T]$, the distance projection $a(t):D_\delta(\Gamma(t))\to\Gamma(t)$ is well-defined, and for sufficiently small $h$ its restriction to $\Gamma_{h,*}(t)\subset D_\delta(\Gamma(t))$ is a bijection onto $\Gamma(t)$. Let $n(\cdot,t)$ denote the unit normal vector on $\Gamma(t)$. For $x\in D_\delta(\Gamma(t))$,
\begin{equation}\label{eq:signed-distance-representation}
	x-a(t)(x) = d(x,t)n(a(t)(x),t),
\end{equation}
where $d(\cdot,t)$ denotes the signed distance function to $\Gamma(t)$. Let $P=I-n\otimes n$ denote the orthogonal projection onto the tangent space of $\Gamma(t)$.

For a function $u_h$ on $\Gamma_{h,*}(t)$ and a function $f$ on $\Gamma(t)$, we define the lift and inverse lift by
\begin{equation}\label{eq:lift}
	u_h^\ell=u_h\circ a(t)^{-1}, \qquad f^{-\ell}=f\circ a(t).
\end{equation}
In particular, the inverse lifts of the unit normal and of the tangential projection are denoted by
\begin{equation*}
	n_*=n^{-\ell}, \qquad P_*=P^{-\ell}=I-n_*\otimes n_* .
\end{equation*}

We recall the standard lift and interpolation estimates for high-order isoparametric surface finite elements, and the inverse estimates for quasi-uniform triangulations; see \cite[Lemmas~5.1--5.3 and~5.6]{Kovacs}, \cite[Propositions~2.3 and~2.7]{Demlow} and \cite[(3.3a), (4.3c)]{HLT}.
\begin{lemma}\label{lem:approximation}
	For $u_h$ defined on $\Gamma_{h,*}(t)$, $f\in H^{k+1}(\Gamma(t))$, $s = 0, 1$, and $1\leq p \leq \infty$, there exists a constant $C$ independent of $h$ and $t$ such that
	\begin{align}
		C^{-1}\norm{u_h}{W^{s, p}(\Gamma_{h,*}(t))} \le \norm{u_h^\ell}{W^{s, p}(\Gamma(t))} &\le C\norm{u_h}{W^{s, p}(\Gamma_{h,*}(t))}, \label{eq:lift-norm} \\
		\norm{(I_hf^{-\ell})^\ell-f}{L^2(\Gamma(t))}+ h\norm{(I_hf^{-\ell})^\ell-f}{H^1(\Gamma(t))}&\leq Ch^{k+1}\norm{f}{H^{k+1}(\Gamma(t))}. \label{eq:interpolation}
	\end{align}
	Moreover, for $u_h,w_h$ defined on $\Gamma_{h,*}(t)$,
	\begin{align}
		&\left| \int_{\Gamma_{h,*}(t)}u_hw_h - \int_{\Gamma(t)}u_h^\ell w_h^\ell \right|\le Ch^{k+1} \norm{u_h}{L^2(\Gamma_{h,*}(t))} \norm{w_h}{L^2(\Gamma_{h,*}(t))}, \label{eq:lift-mass} \\
		&\left| \int_{\Gamma_{h,*}(t)}u_h \nabla_{\Gamma_{h,*}(t)}w_h - \int_{\Gamma(t)}u_h^\ell \nabla_{\Gamma(t)}w_h^\ell \right|\le Ch^{k} \norm{u_h}{L^\infty(\Gamma_{h,*}(t))} \norm{w_h}{H^1(\Gamma_{h,*}(t))}, \label{eq:lift-mix} \\
		&\left| \int_{\Gamma_{h,*}(t)} \nabla_{\Gamma_{h,*}(t)}u_h\cdot \nabla_{\Gamma_{h,*}(t)}w_h - \int_{\Gamma(t)} \nabla_{\Gamma(t)}u_h^\ell\cdot \nabla_{\Gamma(t)}w_h^\ell \right| \nonumber \\
		&\qquad\le Ch^{k+1} \norm{u_h}{H^1(\Gamma_{h,*}(t))} \norm{w_h}{H^1(\Gamma_{h,*}(t))}. \label{eq:lift-stiffness}
	\end{align}
	For every finite element function $u_h\in S_{h,*}(t)$, $s,r\in\{0,1\}$ with $s\le r$, and $1\le q\le p\le\infty$, the inverse estimate holds
	\begin{equation}\label{eq:inverse}
		\norm{u_h}{W^{r,p}(\Gamma_{h,*}(t))}\le Ch^{s-r-2(1/q-1/p)}\norm{u_h}{W^{s,q}(\Gamma_{h,*}(t))}.
	\end{equation}
\end{lemma}

\begin{lemma}[Superapproximation]\label{lem:superapproximation}
	Let $u_h,f_h\in S_{h,*}(t)$. Then
	\begin{subequations}\label{eq:superapproximation}
		\begin{align}
			&\norm{(1-I_h)(f_hu_h)}{L^2(\Gamma_{h,*}(t))} +h\norm{(1-I_h)(f_hu_h)}{H^1(\Gamma_{h,*}(t))} \nonumber \\
			&\qquad\le Ch^2 \norm{f_h}{W^{1,\infty}(\Gamma_{h,*}(t))} \norm{u_h}{H^1(\Gamma_{h,*}(t))}, \label{eq:FE-superapproximation} \\
			&\norm{(1-I_h)(P_*u_h)}{L^2(\Gamma_{h,*}(t))} +h\norm{(1-I_h)(P_*u_h)}{H^1(\Gamma_{h,*}(t))} \nonumber \\
			&\qquad\le Ch^2\norm{u_h}{H^1(\Gamma_{h,*}(t))}. \label{eq:P-superapproximation}
		\end{align}
	\end{subequations}

	For $f\in W^{k+1,\infty}(\Gamma(t))$,
	\begin{equation}\label{eq:smooth-superapproximation}
		\begin{aligned}
			&\norm{(1-I_h)(f^{-\ell}u_h)}{L^2(\Gamma_{h,*}(t))} +h\norm{(1-I_h)(f^{-\ell}u_h)}{H^1(\Gamma_{h,*}(t))} \\
			&\qquad\le Ch^2\norm{f}{W^{k+1,\infty}(\Gamma(t))} \norm{u_h}{H^1(\Gamma_{h,*}(t))}.
		\end{aligned}
	\end{equation}
\end{lemma}

Here \eqref{eq:FE-superapproximation}--\eqref{eq:P-superapproximation} are from \cite[Lemmas~3.8--3.9]{HLT}, and \eqref{eq:smooth-superapproximation} follows from \cite[Lemma~4.4]{BaiLi}.

\subsection{Perturbed surfaces and normal vectors}\label{subsec:normals}

In this subsection the time variable is fixed and omitted. Let $X_h$ be another regular finite element parametrization and set
\begin{equation*}
	\Gamma_h=X_h(\Gamma_h^0), \qquad S_h=S_h(\Gamma_h) = \{u_h\circ X_h^{-1}:u_h\in S_h^0\}.
\end{equation*}
For any $u_h \in S_h$, define its lift to $\Gamma_{h,*}$ by $\widehat u_h=u_h\circ X_h\circ(X_{h,*})^{-1}\in S_{h,*}$. Let $e_x\in(S_{h,*})^3$ denote the finite element function whose nodal vector is $X_h-X_{h,*}$. If $\norm{\nabla_{\Gamma_{h,*}}e_x}{L^\infty(\Gamma_{h,*})}\le\frac12$, then by \cite[Lemma~4.3]{KLLP}, for $s=0,1$ and $1\le p\le\infty$,
\begin{equation}\label{eq:transfer-norms}
	C^{-1}\norm{u_h}{W^{s,p}(\Gamma_h)} \le \norm{\widehat u_h}{W^{s,p}(\Gamma_{h,*})} \le C\norm{u_h}{W^{s,p}(\Gamma_h)}.
\end{equation}

Let $\Pi_{h,*}:L^2(\Gamma_{h,*})\to S_{h,*}$ denote the $L^2$ projection, defined by
\begin{equation*}
	\int_{\Gamma_{h,*}}\Pi_{h,*}f\,\chi_h\, = \int_{\Gamma_{h,*}}f\chi_h\, \qquad \forall\chi_h\in S_{h,*}.
\end{equation*}
Analogously, let $\Pi_h:L^2(\Gamma_h)\to S_h$ denote the $L^2$ projection on $\Gamma_h$. Both projections are stable in $L^p$, $1\le p\le\infty$; see \cite[Lemma~3.3]{HLT}.

Let $n_{h,*}$ and $n_h$ denote the geometric unit normals to $\Gamma_{h,*}$ and $\Gamma_h$, respectively, and define their $L^2$ projections by
\begin{equation*}
	\overline n_{h,*}=\Pi_{h,*}n_{h,*}, \qquad \overline n_h=\Pi_h n_h.
\end{equation*}
The corresponding normal and tangential projection matrices are
\begin{equation*}
	\overline N_{h,*} = \frac{\overline n_{h,*}\otimes\overline n_{h,*}}{|\overline n_{h,*}|^2}, \qquad \overline T_{h,*}=I-\overline N_{h,*}.
\end{equation*}
These definitions follow \cite[Section~3.1]{HLT}. For the normal vectors defined above, we have the following estimates:
\begin{lemma}[{\cite[Lemmas~3.4--3.5]{HLT}}]\label{lem:normal-estimates}
	The following estimates hold:
	\begin{subequations}\label{eq:normal-estimates}
		\begin{align}
			\norm{\overline{n}_{h,*} - \widehat{\overline n}_h }{L^2(\Gamma_{h,*})} + h \norm{\overline{n}_{h,*} - \widehat{\overline n}_h }{L^\infty(\Gamma_{h,*})} &\le C\norm{e_x}{H^1(\Gamma_{h,*})}, \\
			\norm{\overline{n}_{h,*} - n_{h,*}}{L^2(\Gamma_{h,*})} + h \norm{\overline{n}_{h,*} - n_{h,*}}{L^\infty(\Gamma_{h,*})} &\le C h^k, \\
			\norm{\overline{n}_{h,*}-n_*}{L^2(\Gamma_{h,*})} + h \norm{\overline{n}_{h,*}-n_*}{L^\infty(\Gamma_{h,*})} &\le Ch^k.
		\end{align}
	\end{subequations}
\end{lemma}
Moreover, if $\norm{\nabla_{\Gamma_{h,*}}e_x}{L^\infty(\Gamma_{h,*})}\le\frac12$, then $\widehat n_h=n_h\circ X_h\circ(X_{h,*})^{-1}$ satisfies
\begin{equation}\label{eq:normal-perturbation}
	\norm{\widehat n_h-n_{h,*}}{L^\infty(\Gamma_{h,*})}\le C\norm{\nabla_{\Gamma_{h,*}}e_x}{L^\infty(\Gamma_{h,*})}.
\end{equation}

For the normal projection operator, we have the following superapproximation estimate.
\begin{lemma}[{\cite[Lemma~3.9]{HLT}}]\label{lem:superapproximation-II}
	For any $u_h\in(S_{h,*})^3$,
	\begin{equation}\label{eq:superapproximation-II}
		\norm{(1-I_h)(\overline N_{h,*} u_h)}{L^2(\Gamma_{h,*})} + h \norm{(1-I_h)(\overline N_{h,*} u_h)}{H^1(\Gamma_{h,*})} \le C h^2\norm{u_h}{H^1(\Gamma_{h,*})}.
	\end{equation}
\end{lemma}

\subsection{The BGN method and the main theorem}\label{subsec:BGN}

Let $t_m = m\tau$, $X_0=\Id_{\Gamma^0}$, and $X_h^0 = I_hX_0$, so that $\Gamma_h^0 = X_h^0(\Gamma_h^0)$. The exactly integrated BGN method \cite[(3.1a)--(3.1b)]{GNPZ} is given as follows: Find $V_h^{m+1}\in S_h(\Gamma_h^m)^3$ and $H_h^{m+1}\in S_h(\Gamma_h^m)$ such that
\begin{subequations}\label{eq:BGN-mixed}
	\begin{align}
		\int_{\Gamma_h^m} (V_h^{m+1}\cdot n_h^m)\chi_h +\int_{\Gamma_h^m} H_h^{m+1}\chi_h&= 0, \quad \forall\chi_h\in S_h(\Gamma_h^m), \label{eq:BGN-mixed-normal} \\
		\int_{\Gamma_h^m} H_h^{m+1}n_h^m\cdot\phi_h-\int_{\Gamma_h^m} \nabla_{\Gamma_h^m}(\Id+\tau V_h^{m+1})\cdot \nabla_{\Gamma_h^m}\phi_h &=0, \quad \forall\phi_h\in S_h(\Gamma_h^m)^3. \label{eq:BGN-mixed-vector}
	\end{align}
\end{subequations}
Here $\Gamma_h^m = X_h^m(\Gamma_h^0)$, and the nodal values are updated by $X_h^{m+1}=X_h^m+\tau V_h^{m+1}$.

The first equation in \eqref{eq:BGN-mixed} gives
\begin{equation*}
	H_h^{m+1} = -\Pi_h^m(V_h^{m+1}\cdot n_h^m).
\end{equation*}
Substitution into the second equation yields the equivalent formulation
\begin{equation}\label{eq:BGN-eliminated}
	\int_{\Gamma_h^m} \Pi_h^m(V_h^{m+1}\cdot n_h^m) \phi_h\cdot n_h^m  + \int_{\Gamma_h^m} \nabla_{\Gamma_h^m}(\Id+\tau V_h^{m+1})\cdot \nabla_{\Gamma_h^m}\phi_h =0, 
\end{equation}
for all $\phi_h\in S_h(\Gamma_h^m)^3$. Our analysis is based on this formulation.

For closed embedded surfaces $\Gamma_i=\partial\Omega_i$, $i=1,2$, we measure their distance by the manifold distance \cite{ZJB21}:
\begin{equation*}
	\dman(\Gamma_1,\Gamma_2) = |\Omega_1\triangle\Omega_2|.
\end{equation*}

\begin{theorem}\label{thm:main}
	Let $k\ge6$ and suppose that $\tau=O(h^2)$, i.e., $c_0h^2\le\tau\le C_0h^2$ with constants $c_0,C_0>0$ independent of $h$. Suppose further that the equations \eqref{eq:approximation-flows} of the approximation flows admit solutions $Y^{[0]}$ and $Y^{[1]}$ satisfying Assumption~\ref{ass:regularity} with initial value $X_0$.

	Then, for sufficiently small $h$, the surfaces $\Gamma_h^m$ computed by the BGN method \eqref{eq:BGN-mixed} satisfy
	\begin{equation}\label{eq:main-manifold}
		\max_{t_m\le T} \dman\bigl( \Gamma_h^m,\Gamma[X(t_m)] \bigr) \le C\bigl(\tau+h^4+h^{k+1}\bigr) \le Ch^2.
	\end{equation}
	Here $X(t)$ evolves by mean curvature flow $V = -Hn$ with initial value $X_0$, and the constant $C$ is independent of $h$, $\tau$, and $m$.
\end{theorem}

\section{Approximation flows and the improved defect}\label{sec:approximation-flows}
Throughout this section, $k\ge6$ and we assume that the embedding $Y$ satisfies the following regularity assumption.
\begin{assumption}[Regularity]\label{ass:regularity}
	Let $Y\in C^3([0,T];H^2(\Gamma^0)^3)$, $Y:\Gamma^0\times[0,T]\to\R^3$, be a regular embedding for which there exist constants $M,c>0$ such that
	\begin{equation}\label{eq:regularity}
		\begin{aligned}
			&\sup_{0\le t\le T} \Bigl( \norm{Y(t)}{W^{k+2,\infty}(\Gamma^0)} +\norm{Y_t(t)}{W^{k,\infty}(\Gamma^0)} +\norm{Y_{tt}(t)}{H^2(\Gamma^0)} +\norm{Y_{ttt}(t)}{H^2(\Gamma^0)} \Bigr) \le M, \\
			&\inf_{0\le t\le T}|\nabla_{\Gamma^0}Y(p,t)| \ge c, \qquad p\in\Gamma^0.
		\end{aligned}
	\end{equation}
\end{assumption}

\subsection{The approximation flows}

For a regular embedding $Y$, set $\Gamma[Y]=Y(\Gamma^0)$, and denote by $n[Y]$, $H[Y]$, $A[Y]=\nabla_{\Gamma[Y]}n[Y]$ and $P[Y]=I-n[Y]\otimes n[Y]$ the unit normal, mean curvature, Weingarten map, and tangential projection on $\Gamma[Y]$, respectively. Let $V=Y_t\circ Y^{-1}$ be the velocity. For a field $f$ on $\Gamma[Y]$, its material derivative is $\partial^\bullet f = \partial_t(f\circ Y)\circ Y^{-1}$. The commutator formula for $[\partial^\bullet, \nabla_\Gamma]$ is \cite[Lemma~38(i)]{BGN20}
\begin{equation}\label{eq:commutator}
	\partial^\bullet(\nabla_\Gamma f)=\nabla_\Gamma(\partial^\bullet f)-\bigl(\nabla_\Gamma V-n\otimes n\,(\nabla_\Gamma V)^T\bigr)\nabla_\Gamma f.
\end{equation}
The standard evolution identities \cite[Lemmas~37(i) and~39(i)]{BGN20} give
\begin{subequations}\label{eq:material-identities}
	\begin{align}
		\partial^\bullet n[Y] &= -(\nabla_\Gamma V)n[Y], \\
		\partial^\bullet H[Y] &= -\Delta_\Gamma(V\cdot n[Y]) -|A[Y]|^2(V\cdot n[Y]) +V\cdot\nabla_\Gamma H[Y].
	\end{align}
\end{subequations}
Moreover, the strong form of the transport formula \cite[Lemma~5.5]{Kovacs} reads
\begin{equation}\label{eq:Laplacian-commutator}
	\partial^\bullet(\Delta_\Gamma f) = \Delta_\Gamma\partial^\bullet f + \nabla_{\Gamma}\cdot\bigl(\mathcal E_Y(V)\nabla_\Gamma f\bigr) - (\nabla_{\Gamma}\cdot V)\Delta_\Gamma f,
\end{equation}
where $\mathcal E_Y(V) = (\nabla_{\Gamma}\cdot V)P[Y] - P[Y]\bigl(\nabla_\Gamma V+(\nabla_\Gamma V)^T\bigr)P[Y]$.

For $t+\tau\le T$, define the averaged velocity by
\begin{equation}\label{eq:averaged-velocity}
	\overline V_\tau[Y]=\frac{Y(t+\tau)-Y(t)}\tau\circ Y(t)^{-1}.
\end{equation}
Let $Y_{h,*}(t)=I_hY(t)$ and $\Gamma_{h,*}(t)=Y_{h,*}(t)(\Gamma_h^0)$; the notation of Section~\ref{sec:discretization} (lifts, $n_*$, $P_*$, $n_{h,*}$, $\Pi_{h,*}$, \dots) is used with $X$ replaced by $Y$. Its interpolation is
\begin{equation*}
	\overline V_{h,*}(t) = \frac{Y_{h,*}(t+\tau)-Y_{h,*}(t)}{\tau} \circ Y_{h,*}(t)^{-1}.
\end{equation*}

The continuous BGN defect of \eqref{eq:averaged-velocity} on $\Gamma[Y(t)]$ is defined by
\begin{equation}\label{eq:continuous-defect}
	r_\tau[Y]=(\overline V_\tau[Y]\cdot n)n-\lap(\Id+\tau\overline V_\tau[Y]), \qquad Pr_\tau[Y]=-\tau P\lap\overline V_\tau[Y],
\end{equation}
where the second identity follows from $P\lap\Id=-HPn=0$.

We construct the approximation flows by an expansion in time. Formally write
\begin{equation*}
	Y_t\circ Y^{-1}=V_0[Y]+\tau V_1[Y]+\tau^2V_2[Y]+\cdots,
\end{equation*}
it should make the defect $r_\tau[Y]$ to achieve a high order in $\tau$. 

We now define the correction terms $W_0, V_0, \alpha_1, W_1, V_1$ as functionals of $Y$.
\begin{subequations}\label{eq:correction-fields}
	\begin{align}
		W_0[Y]\cdot n[Y]&=0, \qquad P[Y]\Delta_\Gamma W_0[Y] = P[Y]\Delta_\Gamma\bigl(H[Y]n[Y]\bigr), \label{eq:W0} \\
		V_0[Y]&=-H[Y]n[Y]+W_0[Y], \nonumber \\
		\alpha_1[Y] &= -\frac12\Delta_\Gamma H[Y] +\frac32H[Y]|A[Y]|^2 \nonumber \\
		&\quad -2A[Y]:\nabla_\Gamma W_0[Y] +\frac12W_0[Y]\cdot A[Y]W_0[Y], \label{eq:alpha1} \\
		W_1[Y]\cdot n[Y]&=0, \nonumber \\
		P[Y]\Delta_\Gamma W_1[Y] &= -P[Y]\Delta_\Gamma\bigl(\alpha_1[Y]n[Y]\bigr) +\frac12P[Y]\nabla_{\Gamma}\cdot \bigl(\mathcal E_Y(V_0[Y])\nabla_\Gamma V_0[Y]\bigr) \nonumber \\
		&\quad -\frac12 \bigl(\Delta_\Gamma V_0[Y]\cdot n[Y]\bigr) \bigl(\nabla_\Gamma H[Y]+A[Y]W_0[Y]\bigr), \label{eq:W1} \\
		V_1[Y]&=\alpha_1[Y]n[Y]+W_1[Y]. \nonumber
	\end{align}
\end{subequations}

The approximation flows are the solutions on $\Gamma^0$ of
\begin{equation}\label{eq:approximation-flows}
	\begin{aligned}
		Y_t^{[0]}&=V_0[Y^{[0]}]\circ Y^{[0]},& Y^{[0]}(0)&=X_0, \\
		Y_t^{[1]}&=(V_0[Y^{[1]}]+\tau V_1[Y^{[1]}])\circ Y^{[1]},& Y^{[1]}(0)&=X_0.
	\end{aligned}
\end{equation}

First we show that $W_0, W_1$ are well-defined.
\begin{lemma}\label{lem:tangent-elliptic}
	Let $2\leq j\leq k$. For every $f \in H^j(\Gamma[Y])$ and $g\in H^{j-2}(\Gamma[Y])^3$ satisfying $P[Y] g = g$, there exists a unique $u\in H^j(\Gamma[Y])^3$ with $u\cdot n[Y]=f$ and $P[Y]\lap u=g$. Moreover,
	\begin{equation}\label{eq:tangent-elliptic-bound}
		\norm{u\circ Y}{H^j(\Gamma^0)} \le C\bigl(\norm{f\circ Y}{H^j(\Gamma^0)} + \norm{g\circ Y}{H^{j-2}(\Gamma^0)}\bigr),
	\end{equation}
	where $C$ depends only on $M$, $c$ and $\Gamma^0$.
\end{lemma}

\begin{proof}
	We omit $[Y]$ for simplicity. Let $\mathcal{L} u= -P\lap u +P\lap\bigl((u\cdot n)n\bigr) +n(1-\lap)(u\cdot n)$. Integration by parts gives
	\begin{equation*}
		\int_{\Gamma[Y]}(\mathcal{L} u)\cdot v = \int_{\Gamma[Y]} \Bigl( \gr(Pu)\cdot\gr(Pv) +\gr(u\cdot n)\cdot\gr(v\cdot n) +(u\cdot n)(v\cdot n) \Bigr),
	\end{equation*}
	which is coercive on $H^1(\Gamma[Y])^3$ by the tangential Poincar\'e inequality $\norm{Pu}{L^2(\Gamma[Y])}\le C\norm{\gr(Pu)}{L^2(\Gamma[Y])}$. By the Lax--Milgram theorem and elliptic regularity, $\mathcal L:H^j(\Gamma[Y])^3\to H^{j-2}(\Gamma[Y])^3$ is an isomorphism. Since $Pg=g$, taking the normal and the tangential components shows that $u$ satisfies $u\cdot n=f$ and $P\lap u=g$ if and only if $\mathcal Lu=-g+P\lap(f n)+n(1-\lap)f$: the normal component is $(1-\lap)(u\cdot n-f)=0$, and then the tangential component reduces to $P\lap u=g$. This gives existence and uniqueness, and $\norm{u}{H^j(\Gamma[Y])}\le C\norm{\mathcal Lu}{H^{j-2}(\Gamma[Y])}\le C(\norm{f}{H^j(\Gamma[Y])}+\norm{g}{H^{j-2}(\Gamma[Y])})$. The norm equivalence along $Y$ yields \eqref{eq:tangent-elliptic-bound}.
\end{proof}

\begin{lemma}[Regularity of the correction terms]\label{lem:field-bounds}
	The problems \eqref{eq:W0} and \eqref{eq:W1} have unique solutions $W_0[Y]\in H^k(\Gamma[Y])^3$ and $W_1[Y]\in H^{k-2}(\Gamma[Y])^3$, and
	\begin{equation}\label{eq:field-bounds}
		\begin{aligned}
			\sup_{0\le t\le T}\Bigl( &\norm{V_0[Y]\circ Y}{H^k(\Gamma^0)} +\norm{V_1[Y]\circ Y}{H^{k-2}(\Gamma^0)} \\
			&+\norm{(\partial^\bullet V_0[Y])\circ Y}{H^{k-2}(\Gamma^0)} +\norm{(\partial^\bullet V_1[Y])\circ Y}{H^{k-4}(\Gamma^0)} \Bigr) \le C,
		\end{aligned}
	\end{equation}
	where $C$ depends only on $M$, $c$ and $\Gamma^0$. The same bounds hold for $W_0[Y], \alpha_1[Y], W_1[Y]$ and their material derivatives.
\end{lemma}

\begin{proof}
	We omit $t$ and $[Y]$ for simplicity. By Assumption~\ref{ass:regularity}, $n\circ Y$ is bounded in $W^{k+1,\infty}(\Gamma^0)$, and $A\circ Y$, $H\circ Y$ and $V\circ Y=Y_t$ are bounded in $W^{k,\infty}(\Gamma^0)$. Below we use that $H^s$, $s\ge2$, is an algebra. \eqref{eq:W0} gives
	\begin{equation*}
		V_0\cdot n=-H,\qquad P\lap V_0=0.
	\end{equation*}
	Lemma~\ref{lem:tangent-elliptic} with $j=k$ gives unique solvability and $\norm{V_0\circ Y}{H^k(\Gamma^0)}\le C\norm{H\circ Y}{H^k(\Gamma^0)}\le C$. The same bound holds for $W_0=V_0+Hn$. Inserting this into \eqref{eq:alpha1} gives $\norm{\alpha_1\circ Y}{H^{k-2}(\Gamma^0)}\le C$.

	Similarly, \eqref{eq:W1} gives
	\begin{equation}\label{eq:V1-problem}
		V_1\cdot n=\alpha_1,\qquad P\lap V_1 = \frac12P\dv\bigl(\mathcal E_Y(V_0)\gr V_0\bigr) -\frac12(\lap V_0\cdot n)\bigl(\gr H+AW_0\bigr).
	\end{equation}
	Since $V_0 \in H^k$, the right-hand side of the second equation is bounded in $H^{k-2}$, and $\alpha_1$ is bounded in $H^{k-2}$. Hence Lemma~\ref{lem:tangent-elliptic} with $j=k-2$ gives unique solvability and $\norm{V_1\circ Y}{H^{k-2}(\Gamma^0)}\le C$. The same bound holds for $W_1=V_1-\alpha_1n$.

	To obtain the bounds for the material derivatives, we note that the identities \eqref{eq:material-identities} and the transport formula for $\gr$ give
	\begin{equation*}
		\norm{(\partial^\bullet n)\circ Y}{W^{k-1,\infty}(\Gamma^0)} +\norm{(\partial^\bullet H)\circ Y}{W^{k-2,\infty}(\Gamma^0)} +\norm{(\partial^\bullet A)\circ Y}{W^{k-2,\infty}(\Gamma^0)} \le C.
	\end{equation*}
	For a vector field $f$ on $\Gamma[Y]$, using $\partial^\bullet P=-\partial^\bullet n\otimes n-n\otimes\partial^\bullet n$, $Pn=0$ and $P\partial^\bullet n=\partial^\bullet n$, we have
	\begin{equation*}
		P\partial^\bullet(P\lap f) =P(\partial^\bullet P)\lap f+P\partial^\bullet(\lap f) =-(\lap f\cdot n)\,\partial^\bullet n+P\partial^\bullet(\lap f),
	\end{equation*}
	and inserting \eqref{eq:Laplacian-commutator} yields
	\begin{equation}\label{eq:PLap-commutator}
		P\lap(\partial^\bullet f) = P\partial^\bullet(P\lap f) +(\lap f\cdot n)\,\partial^\bullet n -P\dv\bigl(\mathcal E_Y(V)\gr f\bigr) +(\dv V)\,P\lap f .
	\end{equation}

	For $f=V_0$, since $P\lap V_0=0$ and $V_0\cdot n=-H$,
	\begin{equation}\label{eq:dV0-identities}
		(\partial^\bullet V_0)\cdot n =-\partial^\bullet H-V_0\cdot\partial^\bullet n, \qquad P\lap(\partial^\bullet V_0) =(\lap V_0\cdot n)\,\partial^\bullet n -P\dv\bigl(\mathcal E_Y(V)\gr V_0\bigr).
	\end{equation}
	The right-hand sides are bounded in $H^{k-2}$, so \eqref{eq:tangent-elliptic-bound} with $j=k-2$ gives $\norm{(\partial^\bullet V_0)\circ Y}{H^{k-2}(\Gamma^0)}\le C$, and the same bound holds for $\partial^\bullet W_0=\partial^\bullet V_0+\partial^\bullet(Hn)$. Differentiating \eqref{eq:alpha1} and using \eqref{eq:Laplacian-commutator}, the transport formula for $\gr$, and the bounds for $\partial^\bullet H$, $\partial^\bullet A$ and $\partial^\bullet W_0$ gives $\norm{(\partial^\bullet\alpha_1)\circ Y}{H^{k-4}(\Gamma^0)}\le C$.

	For $f=V_1$, since $V_1\cdot n=\alpha_1$,
	\begin{equation*}
		\begin{aligned}
			(\partial^\bullet V_1)\cdot n &=\partial^\bullet\alpha_1-V_1\cdot\partial^\bullet n, \\
			P\lap(\partial^\bullet V_1) &= P\partial^\bullet(P\lap V_1) +(\lap V_1\cdot n)\,\partial^\bullet n \\
			&\quad-P\dv\bigl(\mathcal E_Y(V)\gr V_1\bigr) +(\dv V)\,P\lap V_1 .
		\end{aligned}
	\end{equation*}
	Here $\partial^\bullet\alpha_1-V_1\cdot\partial^\bullet n$ is bounded in $H^{k-4}$. Differentiating the right-hand side of \eqref{eq:V1-problem} bounds $P\partial^\bullet(P\lap V_1)$ in $H^{k-4}$, and the remaining terms are bounded in $H^{k-4}$ since $V_1 \in H^{k-2}$. Hence \eqref{eq:tangent-elliptic-bound} with $j=k-4$ gives $\norm{(\partial^\bullet V_1)\circ Y}{H^{k-4}(\Gamma^0)}\le C$. The same bound for $\partial^\bullet W_1$ follows from $W_1=V_1-\alpha_1n$.
\end{proof}

\subsection{The improved defect}\label{subsec:improved-defect}

\begin{lemma}\label{lem:time-defect}
	Under Assumption~\ref{ass:regularity}, the approximation flows \eqref{eq:approximation-flows} satisfy, uniformly for $t+\tau\le T$,
	\begin{subequations}\label{eq:time-defect-bounds}
		\begin{equation}
			\norm{r_\tau[Y^{[0]}]}{L^2(\Gamma[Y^{[0]}(t)])} + \tau^{-1} \norm{P[Y^{[0]}]r_\tau[Y^{[0]}]} {L^2(\Gamma[Y^{[0]}(t)])} \le C\tau,
		\end{equation}
		\begin{equation}
			\norm{r_\tau[Y^{[1]}]}{L^2(\Gamma[Y^{[1]}(t)])} + \tau^{-1} \norm{P[Y^{[1]}]r_\tau[Y^{[1]}]} {L^2(\Gamma[Y^{[1]}(t)])} \le C\tau^2.
		\end{equation}
	\end{subequations}
\end{lemma}

\begin{proof}
	We omit $t$ and $[Y]$ for simplicity. We only prove the estimate for $Y=Y^{[1]}$; the proof for $Y^{[0]}$ is analogous and omitted.

	For $Y=Y^{[1]}$, we have $Y_t\circ Y^{-1}=V_0+\tau V_1$, and hence $Y_{tt}\circ Y^{-1}=\partial^\bullet V_0+\tau\partial^\bullet V_1$. The second-order Taylor formula gives
	\begin{equation}\label{eq:averaged-velocity-Taylor}
		\begin{aligned}
			\overline V_\tau[Y]&=V_0+\tau\Bigl(V_1+\frac12\partial^\bullet V_0\Bigr)+\rho_1, \\
			\rho_1&=\frac{\tau^2}{2}\partial^\bullet V_1 +\frac{1}{2\tau}\int_0^\tau(\tau-s)^2Y_{ttt}(t+s)\circ Y(t)^{-1}\dd s,
		\end{aligned}
	\end{equation}
	and $\norm{\rho_1}{H^2(\Gamma[Y(t)])}\le C\tau^2$ by Lemma~\ref{lem:field-bounds} and Assumption~\ref{ass:regularity}. Using $V_0\cdot n=-H$, $V_1\cdot n=\alpha_1$ and $\lap V_0=(\lap V_0\cdot n)n$, we obtain
	\begin{equation*}
		\begin{aligned}
			r_\tau[Y^{[1]}] &=\tau\Bigl(\alpha_1+\frac12(\partial^\bullet V_0)\cdot n-\lap V_0\cdot n\Bigr)n+(\rho_1\cdot n)n \\
			&\quad -\tau^2\lap\Bigl(V_1+\frac12\partial^\bullet V_0\Bigr) -\tau\lap\rho_1, \\
			Pr_\tau[Y^{[1]}] &=-\tau^2P\lap\Bigl(V_1+\frac12\partial^\bullet V_0\Bigr) -\tau P\lap\rho_1 .
		\end{aligned}
	\end{equation*}
	Since $\norm{\rho_1}{H^2(\Gamma[Y(t)])}\le C\tau^2$, by Lemma~\ref{lem:field-bounds} it suffices to show
	\begin{equation*}
		\Bigl\|\alpha_1+\frac12(\partial^\bullet V_0)\cdot n-\lap V_0\cdot n\Bigr\|_{L^2(\Gamma[Y(t)])} +\Bigl\|P\lap\Bigl(V_1+\frac12\partial^\bullet V_0\Bigr)\Bigr\|_{L^2(\Gamma[Y(t)])} \le C\tau .
	\end{equation*}
	Combining \eqref{eq:dV0-identities} and \eqref{eq:V1-problem} with \eqref{eq:material-identities} gives
	\begin{equation*}
		\begin{aligned}
			&\alpha_1+\frac12(\partial^\bullet V_0)\cdot n-\lap V_0\cdot n \\
			&\quad=\frac{\tau}{2}\bigl(\lap\alpha_1+\alpha_1|A|^2-W_1\cdot\gr H+W_0\cdot\gr\alpha_1-W_0\cdot AW_1\bigr),
		\end{aligned}
	\end{equation*}
	\begin{equation*}
		P\lap\Bigl(V_1+\frac12\partial^\bullet V_0\Bigr)=\frac{\tau}{2}\Bigl((\lap V_0\cdot n)\bigl(-\gr\alpha_1+AW_1\bigr)-P\dv\bigl(\mathcal E_Y(V_1)\gr V_0\bigr)\Bigr).
	\end{equation*}
	The terms multiplied by $\tau$ are bounded in $L^2$ by Lemma~\ref{lem:field-bounds}. This proves \eqref{eq:time-defect-bounds}.
\end{proof}

For $\phi_h\in S_h(\Gamma_{h,*}(t))^3$, the (discrete) BGN defect is
\begin{equation}\label{eq:BGN-defect}
	d_h[Y](\phi_h) = \int_{\Gamma_{h,*}(t)} \Pi_{h,*}(\overline V_{h,*}\cdot n_{h,*})\, (\phi_h\cdot n_{h,*})+ \int_{\Gamma_{h,*}(t)} \nabla_{\Gamma_{h,*}}(\Id+\tau \overline V_{h,*})\cdot \nabla_{\Gamma_{h,*}}\phi_h.
\end{equation}

\begin{proposition}\label{prop:improved-defect}
	Under Assumption~\ref{ass:regularity}, for all $\phi_h\in S_h(\Gamma_{h,*}(t))^3$,
	\begin{align}
		|d_h[Y](\phi_h)| &\le C\bigl(\norm{r_\tau}{L^2(\Gamma[Y(t)])}+h^{k-1}\bigr) \norm{\phi_h}{L^2(\Gamma_{h,*}(t))}, \label{eq:full-defect} \\
		|d_h[Y](I_h(P_*\phi_h))| &\le C\bigl(\norm{Pr_\tau}{L^2(\Gamma[Y(t)])} +h^2\norm{r_\tau}{L^2(\Gamma[Y(t)])}+h^k\bigr) \nonumber \\
		& \qquad \norm{I_h(P_*\phi_h)}{H^1(\Gamma_{h,*}(t))}. \label{eq:tangential-defect}
	\end{align}
\end{proposition}
\begin{proof}
	We omit the time $t$ for simplicity. For $f\in H^{k}(\Gamma)^3$, we split
	\begin{equation*}
		\begin{aligned}
			&\int_{\Gamma_{h,*}}\Pi_{h,*}(I_hf^{-\ell}\cdot n_{h,*})\,(\phi_h\cdot n_{h,*}) -\int_\Gamma(f\cdot n)(\phi_h^\ell\cdot n) \\
			&= \int_{\Gamma_{h,*}}\Pi_{h,*}\bigl(I_hf^{-\ell}\cdot n_{h,*}-(f\cdot n)^{-\ell}\bigr)(\phi_h\cdot n_{h,*}) +\int_{\Gamma_{h,*}}(\Pi_{h,*}-1)(f\cdot n)^{-\ell}\,(\phi_h\cdot n_{h,*}) \\
			&\qquad +\int_{\Gamma_{h,*}}(f\cdot n)^{-\ell}\,\phi_h\cdot(n_{h,*}-n_*) +\Bigl(\int_{\Gamma_{h,*}}(f\cdot n)^{-\ell}(\phi_h\cdot n_*) -\int_\Gamma(f\cdot n)(\phi_h^\ell\cdot n)\Bigr).
		\end{aligned}
	\end{equation*}
	We bound the four terms on the right-hand side by $Ch^k\norm{f}{H^k(\Gamma)}\norm{\phi_h}{L^2(\Gamma_{h,*})}$: the first by the $L^2$ stability of $\Pi_{h,*}$, \eqref{eq:interpolation}, $\norm{n_{h,*}-n_*}{L^2(\Gamma_{h,*})}\le Ch^k$ from \eqref{eq:normal-estimates} and $H^2(\Gamma)\subset L^\infty(\Gamma)$; the second by the best approximation property of $\Pi_{h,*}$ and \eqref{eq:interpolation}; the third again by $\norm{n_{h,*}-n_*}{L^2(\Gamma_{h,*})}\le Ch^k$; and the last by \eqref{eq:lift-mass}. Hence
	\begin{equation*}
		\left|\int_{\Gamma_{h,*}}\Pi_{h,*}(I_hf^{-\ell}\cdot n_{h,*})\,(\phi_h\cdot n_{h,*}) -\int_\Gamma(f\cdot n)(\phi_h^\ell\cdot n)\right| \le Ch^k\norm{f}{H^{k}(\Gamma)}\norm{\phi_h}{L^2(\Gamma_{h,*})}.
	\end{equation*}
	For $f\in H^{k+1}(\Gamma)^3$, \eqref{eq:lift-stiffness} for the geometric error and \eqref{eq:lift-norm}, \eqref{eq:interpolation} for $\norm{(I_hf^{-\ell})^\ell-f}{H^1(\Gamma)}$ give
	\begin{equation*}
		\left|\int_{\Gamma_{h,*}}\nabla_{\Gamma_{h,*}}I_hf^{-\ell}\cdot\nabla_{\Gamma_{h,*}}\phi_h -\int_\Gamma\nabla_\Gamma f\cdot\nabla_\Gamma\phi_h^\ell\right| \le Ch^k\norm{f}{H^{k+1}(\Gamma)}\norm{\phi_h}{H^1(\Gamma_{h,*})}.
	\end{equation*}
	Since $\overline V_\tau\circ Y(t)=\int_0^1Y_t(t+s\tau)\dd s$, \eqref{eq:regularity} gives $\norm{\overline V_\tau\circ Y(t)}{W^{k,\infty}(\Gamma^0)}\le M$, so $\overline V_\tau$ is bounded in $H^{k}(\Gamma)$ uniformly in $\tau$. Moreover, $(\Id+\tau\overline V_\tau)\circ Y(t)=Y(t+\tau)$ is bounded in $W^{k+2,\infty}(\Gamma^0)$, so $\Id+\tau\overline V_\tau$ is bounded in $H^{k+1}(\Gamma)$. Applying the first estimate to $f=\overline V_\tau$ and the second to $f=\Id+\tau\overline V_\tau$ gives
	\begin{equation}\label{eq:spatial-defect}
		\left|d_h[Y](\phi_h)-\int_\Gamma r_\tau\cdot\phi_h^\ell\right| \le Ch^k\norm{\phi_h}{H^1(\Gamma_{h,*})}.
	\end{equation}
	The inverse estimate \eqref{eq:inverse} proves \eqref{eq:full-defect}.

	Since $P^2=P$, we have $I_h(P_*\phi_h)=I_h\bigl(P_*I_h(P_*\phi_h)\bigr)$, and therefore $ (I-P_*)I_h(P_*\phi_h) =-(1-I_h)\bigl(P_*I_h(P_*\phi_h)\bigr)$. Hence \eqref{eq:lift-norm} and superapproximation \eqref{eq:P-superapproximation} give
	\begin{equation}\label{eq:superapproximation-defect}
		\norm{(I-P)\bigl(I_h(P_*\phi_h)\bigr)^\ell}{L^2(\Gamma)}\le Ch^2\norm{I_h(P_*\phi_h)}{H^1(\Gamma_{h,*})}.
	\end{equation}
	Since $r_\tau=Pr_\tau+(r_\tau\cdot n)n$ and $n\cdot P=0$, \eqref{eq:superapproximation-defect} and the norm equivalence \eqref{eq:lift-norm} give
	\begin{equation*}
		\begin{aligned}
			&\left|\int_\Gamma r_\tau\cdot\bigl(I_h(P_*\phi_h)\bigr)^\ell\right| \\
			& \leq \left|\int_\Gamma Pr_\tau\cdot\bigl(I_h(P_*\phi_h)\bigr)^\ell \right| + \left|\int_\Gamma(r_\tau\cdot n)\,n\cdot(I-P)\bigl(I_h(P_*\phi_h)\bigr)^\ell \right| \\
			&\le C\norm{Pr_\tau}{L^2(\Gamma)}\norm{I_h(P_*\phi_h)}{L^2(\Gamma_{h,*})} +Ch^2\norm{r_\tau}{L^2(\Gamma)}\norm{I_h(P_*\phi_h)}{H^1(\Gamma_{h,*})}.
		\end{aligned}
	\end{equation*}
	Combining this with \eqref{eq:spatial-defect} for $I_h(P_*\phi_h)$ in place of $\phi_h$ proves \eqref{eq:tangential-defect}.
\end{proof}

\begin{lemma}[Comparison of the approximation flows]\label{lem:flow-comparison}
	If $Y^{[0]}$ and $Y^{[1]}$ satisfy Assumption~\ref{ass:regularity}, then
	\begin{equation}\label{eq:flow-comparison}
		\sup_{0\le t\le T}\norm{Y^{[1]}(t)-Y^{[0]}(t)}{H^1(\Gamma^0)}\le C\tau.
	\end{equation}
\end{lemma}

The estimate \eqref{eq:flow-comparison} follows from the argument of \cite[Lemma~4.12]{Li} for curves; we omit the proof here. The proof is also parallel to the convergence proof of the discrete problem, which is carried out in detail in Section~\ref{sec:proof}.

\section{Proof of the main theorem}\label{sec:proof}

Throughout this section, $k\ge6$, $c_0h^2\le\tau\le C_0h^2$, the approximation flows $Y^{[0]}$ and $Y^{[1]}$ satisfy Assumption~\ref{ass:regularity}, and $Y=Y^{[1]}$. $\Gamma=\Gamma[Y(t)]$ is a smooth surface with $n$, $H$, $A$, $P$, $\Gamma_{h,*}=Y_{h,*}(t)(\Gamma_h^0)$ is its interpolated surface, and $X_h$, $\Gamma_h$, $e_x$, $n_h$, $\overline n_h$, $\Pi_h$ are as in Section~\ref{subsec:normals}. Moreover, $\overline V_\tau=\overline V_\tau[Y](t)$, and $\overline V_{h,*}=\overline V_{h,*}(t)=I_h\overline V_\tau^{-\ell}$ is its interpolation. If no confusion arises, for $u_h \in S_h(\Gamma_h)$ we also write its lift $\widehat{u}_h \in S_h(\Gamma_{h,*})$ simply as $u_h$, omitting the hat.

Following \cite{BaiLi}, for a surface $\Gamma$ with unit normal $n_\Gamma$,
\begin{equation}\label{eq:BL-forms}
	\begin{aligned}
		A_\Gamma(u,w)&=\int_\Gamma\nabla_\Gamma u\cdot \nabla_\Gamma w, \\
		A_\Gamma^N(u,w)&=\int_\Gamma[(\nabla_\Gamma u)n_\Gamma] \cdot[(\nabla_\Gamma w)n_\Gamma], \\
		A_\Gamma^T(u,w)&=\int_\Gamma \tr[(\nabla_\Gamma u)(I - n_\Gamma \otimes n_\Gamma)(\nabla_\Gamma w)^T], \\
		B_\Gamma(u,w)&=\int_\Gamma\bigl[(\nabla_\Gamma\cdot u)(\nabla_\Gamma\cdot w) -\tr\bigl((\nabla_\Gamma u)(\nabla_\Gamma w)\bigr)\bigr].
	\end{aligned}
\end{equation}
It holds that $A_\Gamma^T(u,w)=A_\Gamma(u,w)-A_\Gamma^N(u,w)$. We write $A_{h,*}$, $A_{h,*}^N$, $A_{h,*}^T$ for the forms on $\Gamma_{h,*}$, and $B(u_h,w_h)=B_{\Gamma}(u_h^\ell,w_h^\ell)$. Furthermore, we denote by $J$ the perturbation of the mass term,
\begin{equation}\label{eq:J-definition}
	J(\phi_h) =\int_{\Gamma_h}\Pi_h(\overline V_{h,*}\cdot n_h)\,(\phi_h\cdot n_h) -\int_{\Gamma_{h,*}}\Pi_{h,*}(\overline V_{h,*}\cdot n_{h,*})\,(\phi_h\cdot n_{h,*}).
\end{equation}
If $\norm{\nabla_{\Gamma_{h,*}}e_x}{L^\infty(\Gamma_{h,*})}\le\frac12$, then from the argument of \cite[(5.17)--(5.21)]{HLT}
\begin{equation}\label{eq:J-bound}
	|J(\phi_h)|\le C\norm{e_x}{H^1(\Gamma_{h,*})}\norm{\phi_h}{L^2(\Gamma_{h,*})}.
\end{equation}
See also \eqref{eq:J-leading} in Appendix~\ref{app:tangential-identity} for a direct proof.

\subsection{An algebraic identity for tangential test functions}\label{subsec:tangential-identity}

On a smooth surface, $\int_\Gamma \nabla_{\Gamma}\Id \cdot \nabla_\Gamma \phi=\int_\Gamma Hn\cdot \phi$ for every $\phi$. Its derivative along any perturbation $e$ gives an algebraic identity for test functions, i.e.,
\begin{equation*}
	\begin{aligned}
		A_\Gamma^N(e,\phi)+B_\Gamma(e,\phi) &=\int_\Gamma\nabla_\Gamma(e\cdot n)\cdot\nabla_\Gamma(\phi\cdot n)-\int_\Gamma H\,[(\nabla_\Gamma e)n]\cdot\phi \\
		&\quad +\int_\Gamma\left(-|A|^2(e\cdot n)+e\cdot\nabla_\Gamma H+H\nabla_\Gamma\cdot e\right)(\phi\cdot n).
	\end{aligned}
\end{equation*}
The following lemma is a discrete version along the perturbation $e_x$. We first consider a general test function. The proof is given in Appendix~\ref{app:tangential-identity}.
\begin{lemma}\label{lem:tangential-identity}
	Assume that $\norm{e_x}{H^1(\Gamma_{h,*})}\le h^{7/2}$. Then for every $\phi_h\in S_h(\Gamma_h)^3$,
	\begin{equation}\label{eq:T-T}
		\begin{aligned}
			&\left|J(\phi_h)+A_{h,*}^N(e_x,\phi_h)+B(e_x,\phi_h)-\int_{\Gamma}\nabla_{\Gamma}\bigl(e_x^\ell\cdot n\bigr)\cdot\nabla_{\Gamma}(\phi_h^\ell\cdot n)\right| \\
			&\quad \le C\norm{e_x}{H^1(\Gamma_{h,*})}\left(\norm{\phi_h^\ell\cdot n}{L^2(\Gamma)} +(\tau+h^{5/2} + h^{k-1})\norm{\phi_h}{H^1(\Gamma_{h,*})}\right).
		\end{aligned}
	\end{equation}
\end{lemma}

The following corollary restricts Lemma~\ref{lem:tangential-identity} to tangential test functions; it is the analogue of the curve scenario \cite[Lemma~3.6]{Li}.
\begin{corollary}\label{cor:tangential-identity}
	Under the assumptions of Lemma~\ref{lem:tangential-identity}, for every $\phi_h\in S_h(\Gamma_h)^3$,
	\begin{equation}\label{eq:T-T-tangential}
		\begin{aligned}
			&\left|J(I_h(\overline T_{h,*}\phi_h))+A_{h,*}^N(e_x,I_h(\overline T_{h,*}\phi_h))+B(e_x,I_h(\overline T_{h,*}\phi_h))\right| \\
			&\le C(\tau+h^2)\norm{e_x}{H^1(\Gamma_{h,*})}\norm{I_h(\overline T_{h,*}\phi_h)}{H^1(\Gamma_{h,*})} +Ch^2D_h\norm{I_h(\overline T_{h,*}\phi_h)}{H^1(\Gamma_{h,*})},
		\end{aligned}
	\end{equation}
	where
	\begin{equation}\label{eq:def-of-D}
		D_h:=\sup_{0\neq\chi_h\in S_h(\Gamma_{h,*})} \frac{1}{\norm{\chi_h^\ell}{L^2(\Gamma)}}\int_{\Gamma}\nabla_{\Gamma}(e_x^\ell\cdot n)\cdot\nabla_{\Gamma}\chi_h^\ell .
	\end{equation}
\end{corollary}

\begin{proof}
	Since $P_*n_*=0$ and $I_h(\overline T_{h,*}\phi_h)=I_h(\overline T_{h,*}I_h(\overline T_{h,*}\phi_h))$, we have
	\begin{equation*}
		\begin{aligned}
			I_h(\overline T_{h,*}\phi_h)\cdot n_* &=I_h\bigl(P_*I_h(\overline T_{h,*}\phi_h)\bigr)\cdot n_* +I_h\bigl((\overline T_{h,*}-P_*)I_h(\overline T_{h,*}\phi_h)\bigr)\cdot n_* \\
			&=-\bigl((1-I_h)(P_*I_h(\overline T_{h,*}\phi_h))\bigr)\cdot n_* +I_h\bigl((\overline T_{h,*}-P_*)I_h(\overline T_{h,*}\phi_h)\bigr)\cdot n_* .
		\end{aligned}
	\end{equation*}
	By superapproximation \eqref{eq:P-superapproximation} for the first term, and $\norm{\overline T_{h,*}-P_*}{L^\infty(\Gamma_{h,*})}\le C\norm{\overline n_{h,*}-n_*}{L^\infty(\Gamma_{h,*})}\le Ch^{k-1}$ from Lemma~\ref{lem:normal-estimates} for the second term, we obtain $\norm{I_h(\overline T_{h,*}\phi_h)\cdot n_*}{L^2(\Gamma_{h,*})} + h\norm{I_h(\overline T_{h,*}\phi_h)\cdot n_*}{H^1(\Gamma_{h,*})}\le Ch^2\norm{I_h(\overline T_{h,*}\phi_h)}{H^1(\Gamma_{h,*})}$. Since $\norm{\overline n_{h,*}-n_*}{L^\infty(\Gamma_{h,*})}+\norm{n_h-n_*}{L^\infty(\Gamma_{h,*})}\le Ch^{5/2}$ by Lemma~\ref{lem:normal-estimates}, \eqref{eq:normal-perturbation} and the inverse estimate \eqref{eq:inverse}, the same $L^2$ bound holds with $n_*$ replaced by $\overline n_{h,*}$ or $n_h$. Together with \eqref{eq:lift-norm}, we obtain
	\begin{equation}\label{eq:tangential-smallness}
		\begin{aligned}
			&\norm{(I_h(\overline T_{h,*}\phi_h))^\ell\cdot n}{L^2(\Gamma)} + h\norm{(I_h(\overline T_{h,*}\phi_h))^\ell\cdot n}{H^1(\Gamma)} \\
			&\quad+\norm{I_h(\overline T_{h,*}\phi_h)\cdot\overline n_{h,*}}{L^2(\Gamma_{h,*})}+\norm{I_h(\overline T_{h,*}\phi_h)\cdot n_h}{L^2(\Gamma_{h,*})} \\
			&\le Ch^2\norm{I_h(\overline T_{h,*}\phi_h)}{H^1(\Gamma_{h,*})}.
		\end{aligned}
	\end{equation}

	To bound the integral term in \eqref{eq:T-T}, we consider the Ritz projection $R_h$. For any $u \in H^1(\Gamma)$, its Ritz projection $R_h u$ onto $S_h(\Gamma_{h,*})^\ell$ is the unique solution of
	\begin{equation*}
		\int_\Gamma R_h u\, \chi_h^\ell + \int_\Gamma \nabla_\Gamma R_h u \cdot \nabla_\Gamma \chi_h^\ell = \int_\Gamma u\, \chi_h^\ell + \int_\Gamma \nabla_\Gamma u \cdot \nabla_\Gamma \chi_h^\ell, \quad \forall \chi_h \in S_h(\Gamma_{h,*}).
	\end{equation*}
	The standard Ritz projection estimates (C\'ea's lemma and a duality argument) give
	\begin{equation*}
		\norm{u-R_h u}{H^1(\Gamma)} \le C \norm{u}{H^1(\Gamma)}, \quad \norm{u-R_h u}{L^2(\Gamma)} \le Ch\norm{u}{H^1(\Gamma)},
	\end{equation*}
	and hence $\norm{R_h u}{L^2(\Gamma)} \le C\bigl(\norm{u}{L^2(\Gamma)}+h\norm{u}{H^1(\Gamma)}\bigr)$. Since $e_x^\ell\cdot n=(e_x\cdot n_*)^\ell$, the definition of $R_h$ gives
	\begin{equation*}
		\int_\Gamma\nabla_\Gamma\bigl(I_h(e_x\cdot n_*)\bigr)^\ell\cdot\nabla_\Gamma(u-R_hu) = -\int_\Gamma\bigl(I_h(e_x\cdot n_*)\bigr)^\ell(u-R_hu).
	\end{equation*}
	This, together with superapproximation \eqref{eq:smooth-superapproximation}, gives
	\begin{equation*}
		\begin{aligned}
			&\left|\int_\Gamma\nabla_\Gamma(e_x^\ell\cdot n)\cdot\nabla_\Gamma u\right| \\
			& \leq \left|\int_\Gamma\nabla_\Gamma(e_x^\ell\cdot n)\cdot\nabla_\Gamma R_h u\right| + \left|\int_\Gamma\nabla_\Gamma \bigl((1-I_h)(e_x\cdot n_*)\bigr)^\ell\cdot\nabla_\Gamma (u-R_h u)\right| \\
			&\quad + \left|\int_\Gamma\bigl(I_h(e_x\cdot n_*)\bigr)^\ell(u-R_hu)\right| \\
			& \leq C D_h \left\|R_h u\right\|_{L^2(\Gamma)} + C\norm{\bigl((1-I_h)(e_x\cdot n_*)\bigr)^\ell}{H^1(\Gamma)} \left\|u-R_h u\right\|_{H^1(\Gamma)} \\
			&\quad + C\norm{e_x}{L^2(\Gamma_{h,*})}\norm{u-R_hu}{L^2(\Gamma)} \\
			& \leq C\left(D_h+\norm{e_x}{H^1(\Gamma_{h,*})}\right)\left(\left\|u\right\|_{L^2(\Gamma)}+h\left\|u\right\|_{H^1(\Gamma)}\right).
		\end{aligned}
	\end{equation*}
	Taking $u=(I_h(\overline T_{h,*}\phi_h))^\ell\cdot n$, we obtain from \eqref{eq:tangential-smallness} and \eqref{eq:T-T}
	\begin{equation*}
		\begin{aligned}
			&\left|J(I_h(\overline T_{h,*}\phi_h))+A_{h,*}^N(e_x,I_h(\overline T_{h,*}\phi_h))+B(e_x,I_h(\overline T_{h,*}\phi_h))\right| \\
			&\le C\norm{e_x}{H^1(\Gamma_{h,*})}\Bigl(\norm{(I_h(\overline T_{h,*}\phi_h))^\ell\cdot n}{L^2(\Gamma)}+(\tau+h^{5/2} + h^{k-1})\norm{I_h(\overline T_{h,*}\phi_h)}{H^1(\Gamma_{h,*})}\Bigr) \\
			&\quad+C\bigl(D_h+\norm{e_x}{H^1(\Gamma_{h,*})}\bigr)\Bigl(\norm{(I_h(\overline T_{h,*}\phi_h))^\ell\cdot n}{L^2(\Gamma)}+h\norm{(I_h(\overline T_{h,*}\phi_h))^\ell\cdot n}{H^1(\Gamma)}\Bigr) \\
			&\le C(\tau+h^2)\norm{e_x}{H^1(\Gamma_{h,*})}\norm{I_h(\overline T_{h,*}\phi_h)}{H^1(\Gamma_{h,*})} +Ch^2D_h\norm{I_h(\overline T_{h,*}\phi_h)}{H^1(\Gamma_{h,*})}.
		\end{aligned}
	\end{equation*}
	This proves \eqref{eq:T-T-tangential}.
\end{proof}

\subsection{The error equation and the control of $D_h$}\label{subsec:error-equation}

Let $V_h\in S_h(\Gamma_h)^3$ be the solution of the BGN method \eqref{eq:BGN-eliminated}. For all $\phi_h\in S_h(\Gamma_h)^3$,
\begin{align}
	\int_{\Gamma_h}\Pi_h(V_h\cdot n_h)\,(\phi_h\cdot n_h) +\int_{\Gamma_h}\nabla_{\Gamma_h}(\Id+\tau V_h)\cdot \nabla_{\Gamma_h}\phi_h&=0, \label{eq:BGN-step} \\
	\int_{\Gamma_{h,*}}\Pi_{h,*}(\overline V_{h,*}\cdot n_{h,*})\,(\phi_h\cdot n_{h,*}) +\int_{\Gamma_{h,*}}\nabla_{\Gamma_{h,*}}(\Id+\tau\overline V_{h,*})\cdot \nabla_{\Gamma_{h,*}}\phi_h&=d_h(\phi_h), \label{eq:interpolated-step}
\end{align}
where $d_h=d_h[Y]$. Recall the decomposition of the stiffness term in \cite[(5.10)]{BaiLi}, where the remainder $K$ is defined,
\begin{equation}\label{eq:stiffness-decomposition}
	\begin{aligned}
		&A_{\Gamma_h}(\Id+\tau V_h,\phi_h)-A_{\Gamma_{h,*}}(\Id+\tau\overline V_{h,*},\phi_h) \\
		&\quad=A_{h,*}^N(e_x+\tau e_V,\phi_h)+\tau A_{h,*}^T(e_V,\phi_h) +B(e_x,\phi_h)+K(\phi_h).
	\end{aligned}
\end{equation}
Here $e_V=V_h-\overline V_{h,*}$. Subtracting \eqref{eq:interpolated-step} from \eqref{eq:BGN-step} and using \eqref{eq:J-definition} and \eqref{eq:stiffness-decomposition} gives the error equation
\begin{equation}\label{eq:error-equation}
	\begin{aligned}
		&\int_{\Gamma_h}\Pi_h(e_V\cdot n_h)\,(\phi_h\cdot n_h) +A_{h,*}^N(e_x+\tau e_V,\phi_h)+\tau A_{h,*}^T(e_V,\phi_h) \\
		&\quad=-d_h(\phi_h)-J(\phi_h)-B(e_x,\phi_h)-K(\phi_h).
	\end{aligned}
\end{equation}
If $\norm{e_x}{H^1(\Gamma_{h,*})}\le h^{7/2}$, then $\norm{\nabla_{\Gamma_{h,*}}e_x}{L^\infty(\Gamma_{h,*})}\le Ch^{5/2}\le\frac12$ by \eqref{eq:inverse}, and by \eqref{eq:J-bound} for $J$, by \cite[(5.9)]{BaiLi} for $B$, and by \cite[(5.22)]{BaiLi} for $K$, we have for all $\phi_h\in S_h(\Gamma_h)^3$,
\begin{equation}\label{eq:JKB-bounds}
	\begin{aligned}
		&|J(\phi_h)|+|B(e_x,\phi_h)|\le C\norm{e_x}{H^1(\Gamma_{h,*})}\norm{\phi_h}{L^2(\Gamma_{h,*})}, \\
		&|K(\phi_h)| \le C\norm{e_x}{W^{1, \infty}(\Gamma_{h,*})}\bigl(\norm{e_x}{H^1(\Gamma_{h,*})} +\tau\norm{e_V}{H^1(\Gamma_{h,*})}\bigr) \norm{\phi_h}{H^1(\Gamma_{h,*})} \\
		&\qquad \qquad \qquad + C (\tau + h^{k-1}) \norm{e_x}{H^1(\Gamma_{h,*})}\norm{\phi_h}{H^1(\Gamma_{h,*})}.
	\end{aligned}
\end{equation}
For $d_h(\phi_h)$, denote the normal defect and the tangential defect by $D_1$, $D_2$,
\begin{subequations}
	\begin{align}
		&D_1 = \sup_{0\neq\phi_h\in S_h(\Gamma_{h,*})^3} \frac{|d_h(\phi_h)|}{\norm{\phi_h}{L^2(\Gamma_{h,*})}}, \label{eq:normal-defect-discrete} \\
		&D_2 = \sup_{\phi_h\in S_h(\Gamma_{h,*})^3,\ I_h(P_*\phi_h)\neq0} \frac{|d_h(I_h(P_*\phi_h))|}{\norm{I_h(P_*\phi_h)}{H^1(\Gamma_{h,*})}}, \label{eq:tangential-defect-discrete}
	\end{align}
\end{subequations}
and we set
\begin{equation}\label{eq:delta}
	\delta=D_1+\tau^{-1}D_2 .
\end{equation}

\begin{lemma}\label{lem:D-bound}
	Under the assumptions of Lemma~\ref{lem:tangential-identity},
	\begin{equation}\label{eq:D-bound}
		\begin{aligned}
			D_h&\le C\left(\norm{\Pi_h(e_V\cdot n_h)}{L^2(\Gamma_h)} + \norm{I_h(\overline N_{h,*}e_V)}{L^2(\Gamma_{h,*})} \right) \\
			&\quad + C \left(h\norm{I_h(\overline T_{h,*}e_V)}{H^1(\Gamma_{h,*})} +\norm{e_x}{H^1(\Gamma_{h,*})} +D_1\right).
		\end{aligned}
	\end{equation}
\end{lemma}

\begin{proof}
	By \eqref{eq:T-T}, the error equation \eqref{eq:error-equation}, \eqref{eq:JKB-bounds}, \eqref{eq:normal-defect-discrete}, and \eqref{eq:inverse}, for every $\phi_h\in S_h(\Gamma_h)^3$,
	\begin{equation}\label{eq:normal-test}
		\begin{aligned}
			&\left|\int_{\Gamma}\nabla_{\Gamma}\bigl(e_x^\ell\cdot n\bigr)\cdot\nabla_{\Gamma}(\phi_h^\ell\cdot n)\right| \\
			&\le\left|J(\phi_h)+A_{h,*}^N(e_x,\phi_h)+B(e_x,\phi_h)\right|+C\norm{e_x}{H^1(\Gamma_{h,*})}\norm{\phi_h}{L^2(\Gamma_{h,*})} \\
			&\le C\bigl(\norm{\Pi_h(e_V\cdot n_h)}{L^2(\Gamma_h)}+\tau h^{-1}\norm{e_V}{H^1(\Gamma_{h,*})} +\norm{e_x}{H^1(\Gamma_{h,*})}+D_1\bigr)\norm{\phi_h}{L^2(\Gamma_{h,*})}.
		\end{aligned}
	\end{equation}
	Let $\chi_h\in S_h(\Gamma_{h,*})$ and $\phi_h=I_h(n_*\chi_h)$, so that $\phi_h\cdot n_*- \chi_h= -\bigl((1-I_h)(n_*\chi_h)\bigr)\cdot n_*$. By superapproximation \eqref{eq:smooth-superapproximation} and inverse estimate \eqref{eq:inverse},
	\begin{equation*}
		\begin{aligned}
			&\left|\int_{\Gamma}\nabla_{\Gamma}(e_x^\ell\cdot n)\cdot\nabla_{\Gamma}\chi_h^\ell\right| \\
			&\le\left|\int_{\Gamma}\nabla_{\Gamma}\bigl(e_x^\ell\cdot n\bigr)\cdot\nabla_{\Gamma}(\phi_h^\ell\cdot n)\right| +C\norm{e_x}{H^1(\Gamma_{h,*})}\norm{(1-I_h)(n_*\chi_h)}{H^1(\Gamma_{h,*})} \\
			&\le\left|\int_{\Gamma}\nabla_{\Gamma}\bigl(e_x^\ell\cdot n\bigr)\cdot\nabla_{\Gamma}(\phi_h^\ell\cdot n)\right| +C\norm{e_x}{H^1(\Gamma_{h,*})}\norm{\chi_h}{L^2(\Gamma_{h,*})} .
		\end{aligned}
	\end{equation*}
	Since $\norm{e_V}{H^1(\Gamma_{h,*})} \leq \norm{I_h(\overline T_{h,*}e_V)}{H^1(\Gamma_{h,*})} + \norm{I_h(\overline N_{h,*}e_V)}{H^1(\Gamma_{h,*})}$, $\norm{\phi_h}{L^2(\Gamma_{h,*})}\le C\norm{\chi_h}{L^2(\Gamma_{h,*})}$, $\tau\le C_0h^2$, \eqref{eq:inverse} and \eqref{eq:normal-test} give \eqref{eq:D-bound}.
\end{proof}

\begin{remark}\label{rem:time-step}
	The condition $c_0h^2\le\tau\le C_0h^2$ is natural for the BGN method. When the error equation \eqref{eq:error-equation} is tested with $\psi_h=I_h(\overline T_{h,*}e_V)$, the left-hand-side is bounded below by $c\tau\norm{I_h(\overline T_{h,*}e_V)}{H^1(\Gamma_{h,*})}^2$. By Corollary~\ref{cor:tangential-identity}, the remaining terms are bounded by $C\bigl((\tau+h^2)\norm{e_x}{H^1(\Gamma_{h,*})}+h^2D_h\bigr)\norm{\psi_h}{H^1(\Gamma_{h,*})}$, where the factor $h^2$ comes from superapproximation. Hence
	\begin{equation*}
		\norm{\psi_h}{H^1(\Gamma_{h,*})}\le C\Bigl(1+\frac{h^2}{\tau}\Bigr)\bigl(\norm{e_x}{H^1(\Gamma_{h,*})}+D_h\bigr)+\cdots,
	\end{equation*}
	which is bounded by the error itself only if $h^2\le C\tau$. Conversely, in the proof of Lemma~\ref{lem:D-bound}, the term $A_{h,*}^N(\tau e_V,\phi_h)$ is estimated by the inverse estimate, which produces the factor $\tau h^{-2}$ and requires $\tau\le Ch^2$. 
	
	Moreover, the tangential defect $D_2$ enters through $\delta=D_1+\tau^{-1}D_2$, so that \eqref{eq:effective-defect} gives $\delta\le C(\tau^2+h^{k-1}+\tau^{-1}h^k)$, which is of order $h^4$ for $\tau=O(h^2)$ when $k\ge6$.
\end{remark}

\subsection{Error estimate}\label{subsec:error-estimate}

We now return to the time levels $t_m$. We write $X_{h,*}^m=Y_{h,*}(t_m)=I_hY(t_m)$ and $\Gamma_{h,*}^m=X_{h,*}^m(\Gamma_h^0)$ for the interpolated flow, and $X_h^m$, $\Gamma_h^m$ for the solution of the BGN method. Accordingly, $n_{h,*}^m$, $\overline n_{h,*}^m$, $\overline N_{h,*}^m$, $\overline T_{h,*}^m$ and $n_h^m$, $\Pi_h^m$ denote the quantities of Section~\ref{subsec:normals} on $\Gamma_{h,*}^m$ and $\Gamma_h^m$, respectively. On the smooth surface $\Gamma^m=\Gamma[Y(t_m)]$ with unit normal $n^m$ and tangential projection $P^m$, we write $n_*^m=(n^m)^{-\ell}$ and $P_*^m=I-n_*^m\otimes n_*^m = (P^m)^{-\ell}$. The velocities and the errors are
\begin{gather*}
	V_h^{m+1}=\frac{X_h^{m+1}-X_h^m}\tau,\qquad \overline V_{h,*}^{m+1}=\frac{X_{h,*}^{m+1}-X_{h,*}^m}\tau, \\
	e_x^m=X_h^m-X_{h,*}^m,\qquad e_V^{m+1}=V_h^{m+1}-\overline V_{h,*}^{m+1}.
\end{gather*}
The quantities of Sections~\ref{subsec:tangential-identity}--\ref{subsec:error-equation} at $t=t_m$ with $X_h=X_h^m$ are denoted by $B^m$, $J^m$, $K^m$, $D_h^m$, $d_h^{m+1}$, $D_1^m$, $D_2^m$ and $\delta^m$, and $A_{h,*}$, $A_{h,*}^N$, $A_{h,*}^T$ are the forms on $\Gamma_{h,*}^m$. 

Lemma~\ref{lem:time-defect} and Proposition~\ref{prop:improved-defect} give
\begin{equation}\label{eq:effective-defect}
	D_1^m\le C(\tau^2+h^{k-1})\le Ch^4,\quad D_2^m\le C(\tau^3+\tau^2h^2+h^k)\le Ch^6,\quad \delta^m\le Ch^4.
\end{equation}

We prove the following convergence between $X_h^m$ and the approximation flows.
\begin{lemma}[Error estimate]\label{lem:error-estimate}
	There are $h_0>0$ and $C>0$, independent of $h$, $\tau$ and $m$, such that for $h\le h_0$ it holds
	\begin{equation}\label{eq:error-estimate}
		\max_{t_m\le T}\norm{e_x^m}{H^1(\Gamma_{h,*}^m)}\le Ch^4 .
	\end{equation}
	Consequently,
	\begin{equation}\label{eq:error-estimate-Y0}
		\begin{aligned}
			&\max_{t_m \le T}\Bigl(\norm{(X_h^m)^{\ell_0}-Y^{[0]}(t_m )}{L^2(\Gamma^0)} +h\norm{(X_h^m)^{\ell_0}-Y^{[0]}(t_m )}{H^1(\Gamma^0)}\Bigr) \\
			&\quad\le C(\tau+ h^4+h^{k+1} )\le Ch^2,
		\end{aligned}
	\end{equation}
	where $(\cdot)^{\ell_0}$ denotes the lift from $\Gamma_h^0$ to $\Gamma^0$ by $a^0$.
\end{lemma}

\begin{proof}
	The proof is by induction. We assume
	\begin{equation}\label{eq:induction}
		\norm{e_x^j}{H^1(\Gamma_{h,*}^j)}\le h^{7/2}, \qquad 0\le j\le m,
	\end{equation}
	so that Lemma~\ref{lem:tangential-identity}, Corollary~\ref{cor:tangential-identity} and Lemma~\ref{lem:D-bound} apply at $t_m$. Moreover, \eqref{eq:transfer-norms} holds since $\norm{\nabla_{\Gamma_{h,*}^m}e_x^m}{L^\infty(\Gamma_{h,*}^m)}\le Ch^{5/2}\le\frac12$ for sufficiently small $h$.

	\emph{The normal error.} Taking $\phi_h=e_V^{m+1}$ in \eqref{eq:error-equation}, the mass term equals $\norm{\Pi_h^m(e_V^{m+1}\cdot n_h^m)}{L^2(\Gamma_h^m)}^2$. Since $e_V^{m+1} = I_h(\overline N_{h,*}^me_V^{m+1}) + I_h(\overline T_{h,*}^me_V^{m+1})$, the $L^2$ stability of $\Pi_h^m$, \eqref{eq:tangential-smallness} and \eqref{eq:transfer-norms} give
	\begin{equation*}
		\begin{aligned}
			&\norm{\Pi_h^m\bigl(I_h(\overline N_{h,*}^me_V^{m+1})\cdot n_h^m\bigr)}{L^2(\Gamma_h^m)} \\
			&\le C\norm{\Pi_h^m(e_V^{m+1}\cdot n_h^m)}{L^2(\Gamma_h^m)} +C\norm{I_h(\overline T_{h,*}^me_V^{m+1})\cdot n_h^m}{L^2(\Gamma_h^m)} \\
			&\le C\norm{\Pi_h^m(e_V^{m+1}\cdot n_h^m)}{L^2(\Gamma_h^m)} +Ch^2\norm{I_h(\overline T_{h,*}^me_V^{m+1})}{H^1(\Gamma_{h,*}^m)}.
		\end{aligned}
	\end{equation*}
	Moreover, we use the identity $I_h(\overline N_{h,*}^mI_h(\overline N_{h,*}^me_V^{m+1})) = I_h(\overline N_{h,*}^me_V^{m+1})$ with \eqref{eq:superapproximation-II} for the first inequality below; the best approximation property of $\Pi_h^m$ and \eqref{eq:FE-superapproximation} for $\norm{(1-\Pi_h^m)(I_h(\overline N_{h,*}^me_V^{m+1})\cdot\overline n_{h,*}^m)}{L^2(\Gamma_h^m)}$ in the second; and $\norm{\overline n_{h,*}^m-n_h^m}{L^\infty(\Gamma_{h,*}^m)}\le Ch^{5/2}$ from Lemma~\ref{lem:normal-estimates}, \eqref{eq:normal-perturbation} and \eqref{eq:inverse} with the $L^2$ stability of $\Pi_h^m$ in the third. Together with \eqref{eq:transfer-norms} and \eqref{eq:inverse}, this gives
	\begin{equation*}
		\begin{aligned}
			&\norm{I_h(\overline N_{h,*}^me_V^{m+1})}{L^2(\Gamma_{h,*}^m)} \\
			&\le C\norm{I_h(\overline N_{h,*}^me_V^{m+1})\cdot\overline n_{h,*}^m}{L^2(\Gamma_{h,*}^m)}+Ch\norm{I_h(\overline N_{h,*}^me_V^{m+1})}{L^2(\Gamma_{h,*}^m)} \\
			&\le C\norm{\Pi_h^m(I_h(\overline N_{h,*}^me_V^{m+1})\cdot\overline n_{h,*}^m)}{L^2(\Gamma_h^m)}+Ch\norm{I_h(\overline N_{h,*}^me_V^{m+1})}{L^2(\Gamma_{h,*}^m)} \\
			&\le C\norm{\Pi_h^m(I_h(\overline N_{h,*}^me_V^{m+1})\cdot n_h^m)}{L^2(\Gamma_h^m)} +Ch\norm{I_h(\overline N_{h,*}^me_V^{m+1})}{L^2(\Gamma_{h,*}^m)}.
		\end{aligned}
	\end{equation*}
	Combining these, for sufficiently small $h$ we obtain
	\begin{equation}\label{eq:mass-normal-coercive}
		\norm{\Pi_h^m(e_V^{m+1}\cdot n_h^m)}{L^2(\Gamma_h^m)}^2 \ge c\norm{I_h(\overline N_{h,*}^me_V^{m+1})}{L^2(\Gamma_{h,*}^m)}^2 -Ch^4\norm{I_h(\overline T_{h,*}^me_V^{m+1})}{H^1(\Gamma_{h,*}^m)}^2 .
	\end{equation}

	For the stiffness term, we know $e_x^{m+1}=e_x^m+\tau e_V^{m+1}$, $A_{h,*}^T$ is nonnegative, and $A_{h,*} = A_{h,*}^N + A_{h,*}^T$. Moreover, $\Gamma_{h,*}^{m+1}=(\Id+\tau\overline V_{h,*}^{m+1})(\Gamma_{h,*}^m)$ with $\norm{\overline V_{h,*}^{m+1}}{W^{1,\infty}(\Gamma_{h,*}^m)}\le C$. Hence the transport formula and $\partial^\bullet n=-(\nabla_\Gamma V)n$ \cite[Lemma~5.1]{HLT}, applied along $\Id+s\tau\overline V_{h,*}^{m+1}$, $s\in[0,1]$, give for every finite element function $u_h$
	\begin{equation}\label{eq:surface-change}
		\begin{aligned}
			&\Bigl|\norm{u_h}{L^2(\Gamma_{h,*}^{m+1})}^2-\norm{u_h}{L^2(\Gamma_{h,*}^m)}^2\Bigr| +\Bigl|\norm{u_h}{H^1(\Gamma_{h,*}^{m+1})}^2-\norm{u_h}{H^1(\Gamma_{h,*}^m)}^2\Bigr| \\
			&\quad+\Bigl|A_{\Gamma_{h,*}^{m+1}}^N(u_h,u_h)-A_{\Gamma_{h,*}^m}^N(u_h,u_h)\Bigr| \le C\tau\norm{u_h}{H^1(\Gamma_{h,*}^m)}^2 .
		\end{aligned}
	\end{equation}
	Combining these with \eqref{eq:surface-change} for $u_h=e_x^{m+1}$, we have
	\begin{equation}\label{eq:stiffness-normal-coercive}
		\begin{aligned}
			&A_{h,*}^N(e_x^{m+1},e_V^{m+1})+\tau A_{h,*}^T(e_V^{m+1},e_V^{m+1}) \\
			&=\frac1{2\tau}\bigl[A_{h,*}^N(e_x^{m+1},e_x^{m+1})-A_{h,*}^N(e_x^m,e_x^m)\bigr] \\
			&\quad+\frac\tau2A_{h,*}(e_V^{m+1},e_V^{m+1})+\frac{\tau}{2} A_{h,*}^T(e_V^{m+1},e_V^{m+1}) \\
			&\ge\frac1{2\tau}\bigl[A_{\Gamma_{h,*}^{m+1}}^N(e_x^{m+1},e_x^{m+1}) -A_{\Gamma_{h,*}^m}^N(e_x^m,e_x^m)\bigr] \\
			&\quad+\frac\tau2\norm{\nabla_{\Gamma_{h,*}^m}e_V^{m+1}}{L^2(\Gamma_{h,*}^m)}^2 - C \norm{e_x^{m+1}}{H^1(\Gamma_{h,*}^{m+1})}^2.
		\end{aligned}
	\end{equation}
	By \eqref{eq:normal-defect-discrete}, \eqref{eq:delta}, \eqref{eq:JKB-bounds} and inverse estimate \eqref{eq:inverse},
	\begin{equation}\label{eq:error-normal-remainder}
		\begin{aligned}
			&|d_h^{m+1}(e_V^{m+1})|+|J^m(e_V^{m+1})|+|B^m(e_x^m,e_V^{m+1})|+|K^m(e_V^{m+1})| \\
			&\quad\le C\bigl(\delta^m+\norm{e_x^m}{H^1(\Gamma_{h,*}^m)}\bigr)\norm{e_V^{m+1}}{L^2(\Gamma_{h,*}^m)} +Ch^{5/2}\tau\norm{e_V^{m+1}}{H^1(\Gamma_{h,*}^m)}^2,
		\end{aligned}
	\end{equation}
	and, for sufficiently small $h$, the last term is absorbed by the $L^2$ norm of $e_V^{m+1}$ controlled below.

	Combining \eqref{eq:error-equation} with $\phi_h=e_V^{m+1}$, \eqref{eq:mass-normal-coercive}, \eqref{eq:stiffness-normal-coercive} and \eqref{eq:error-normal-remainder}, we obtain, for every $\varrho>0$,
	\begin{equation}\label{eq:normal-estimate-preliminary}
		\begin{aligned}
			&\frac{1}{2\tau}\bigl[A_{\Gamma_{h,*}^{m+1}}^N(e_x^{m+1},e_x^{m+1})-A_{\Gamma_{h,*}^m}^N(e_x^m,e_x^m)\bigr]+\frac\tau2\norm{\nabla_{\Gamma_{h,*}^m}e_V^{m+1}}{L^2(\Gamma_{h,*}^m)}^2 \\
			&\quad + c \left(\norm{\Pi_h^m(e_V^{m+1}\cdot n_h^m)}{L^2(\Gamma_h^m)}^2 + \norm{I_h(\overline N_{h,*}^me_V^{m+1})}{L^2(\Gamma_{h,*}^m)}^2\right) \\
			&\quad\le\varrho\norm{I_h(\overline T_{h,*}^me_V^{m+1})}{H^1(\Gamma_{h,*}^m)}^2 \\
			&\quad+C_\varrho\left((\delta^m)^2+\norm{e_x^m}{H^1(\Gamma_{h,*}^m)}^2+\norm{e_x^{m+1}}{H^1(\Gamma_{h,*}^{m+1})}^2\right).
		\end{aligned}
	\end{equation}

	\emph{The tangential error.} We take $\phi_h=I_h(\overline T_{h,*}^me_V^{m+1})$ in \eqref{eq:error-equation}:
	\begin{equation}\label{eq:tangential-equation}
		\begin{aligned}
			\tau A_{h,*}\bigl(e_V^{m+1},I_h(\overline T_{h,*}^me_V^{m+1})\bigr)&=-\int_{\Gamma_h^m}\Pi_h^m(e_V^{m+1}\cdot n_h^m)\,\bigl(I_h(\overline T_{h,*}^me_V^{m+1})\cdot n_h^m\bigr) \\
			&\quad-d_h^{m+1}(I_h(\overline T_{h,*}^me_V^{m+1}))-K^m(I_h(\overline T_{h,*}^me_V^{m+1})) \\
			&\quad-\bigl[J^m(I_h(\overline T_{h,*}^me_V^{m+1}))+A_{h,*}^N(e_x^m,I_h(\overline T_{h,*}^me_V^{m+1})) \\
			&\qquad+B^m(e_x^m,I_h(\overline T_{h,*}^me_V^{m+1}))\bigr].
		\end{aligned}
	\end{equation}
	For the right-hand side of \eqref{eq:tangential-equation}, \eqref{eq:JKB-bounds}, \eqref{eq:induction} and \eqref{eq:inverse} bound $K^m$, and \eqref{eq:T-T-tangential} bounds the last term:
	\begin{equation}\label{eq:K-tangential}
		\begin{aligned}
			&|K^m(I_h(\overline T_{h,*}^me_V^{m+1}))| \\
			&\quad\leq C\left(h^2 \norm{e_x^m}{H^1(\Gamma_{h,*}^m)} + h^{7/2} \norm{e_V^{m+1}}{L^2(\Gamma_{h,*}^m)} \right)\norm{I_h(\overline T_{h,*}^me_V^{m+1})}{H^1(\Gamma_{h,*}^m)},
		\end{aligned}
	\end{equation}
	\begin{equation}\label{eq:bracket-tangential}
		\begin{aligned}
			&\left|J^m(I_h(\overline T_{h,*}^me_V^{m+1}))+A_{h,*}^N(e_x^m,I_h(\overline T_{h,*}^me_V^{m+1}))+B^m(e_x^m,I_h(\overline T_{h,*}^me_V^{m+1}))\right| \\
			&\quad\le C h^2 \left(\norm{e_x^m}{H^1(\Gamma_{h,*}^m)} + D_h^m\right)\norm{I_h(\overline T_{h,*}^me_V^{m+1})}{H^1(\Gamma_{h,*}^m)}.
		\end{aligned}
	\end{equation}
	By the tangential Poincar\'e inequality \cite[(3.41)]{HLT} and the normal-tangential cancellation \cite[Lemma~5.3]{HLT}
	\begin{equation}\label{eq:stiffness-tangential-coercive}
		A_{h,*}\bigl(e_V^{m+1},I_h(\overline T_{h,*}^me_V^{m+1})\bigr) \ge c_1\norm{I_h(\overline T_{h,*}^me_V^{m+1})}{H^1(\Gamma_{h,*}^m)}^2 -C\norm{I_h(\overline N_{h,*}^me_V^{m+1})}{L^2(\Gamma_{h,*}^m)}^2.
	\end{equation}
	For the defect, we write $I_h((I-P_*^m)\psi_h)=(\psi_h\cdot n_*^m)n_*^m+(1-I_h)(P_*^m\psi_h)$ with $\psi_h=I_h(\overline T_{h,*}^me_V^{m+1})$. Then \eqref{eq:tangential-smallness}, \eqref{eq:P-superapproximation}, \eqref{eq:normal-defect-discrete}, \eqref{eq:tangential-defect-discrete} and \eqref{eq:delta} give
	\begin{equation}\label{eq:defect-tangential}
		\begin{aligned}
			&|d_h^{m+1}(I_h(\overline T_{h,*}^me_V^{m+1}))| \\
			&\le D_2^m\norm{I_h\bigl(P_*^mI_h(\overline T_{h,*}^me_V^{m+1})\bigr)}{H^1(\Gamma_{h,*}^m)} +D_1^m\norm{I_h\bigl((I-P_*^m)I_h(\overline T_{h,*}^me_V^{m+1})\bigr)}{L^2(\Gamma_{h,*}^m)} \\
			&\le C(D_2^m+h^2D_1^m)\norm{I_h(\overline T_{h,*}^me_V^{m+1})}{H^1(\Gamma_{h,*}^m)} \\
			&\leq C h^2 \delta^m\norm{I_h(\overline T_{h,*}^me_V^{m+1})}{H^1(\Gamma_{h,*}^m)}.
		\end{aligned}
	\end{equation}

	By \eqref{eq:tangential-smallness} and \eqref{eq:transfer-norms}, the mass term is bounded by
	\begin{equation*}
		Ch^2\norm{\Pi_h^m(e_V^{m+1}\cdot n_h^m)}{L^2(\Gamma_h^m)} \norm{I_h(\overline T_{h,*}^me_V^{m+1})}{H^1(\Gamma_{h,*}^m)}.
	\end{equation*}
	Altogether, using \eqref{eq:D-bound} for $D_h^m$, we obtain for sufficiently small $h$
	\begin{equation}\label{eq:tangential-estimate}
		\begin{aligned}
			\norm{I_h(\overline T_{h,*}^me_V^{m+1})}{H^1(\Gamma_{h,*}^m)}^2&\le C\left(\norm{\Pi_h^m(e_V^{m+1}\cdot n_h^m)}{L^2(\Gamma_h^m)}^2 + \norm{I_h(\overline N_{h,*}^me_V^{m+1})}{L^2(\Gamma_{h,*}^m)}^2\right) \\
			&\quad +C\left(\left\|e_x^m\right\|_{H^1(\Gamma_{h,*}^m)}^2 + (\delta^m)^2\right).
		\end{aligned}
	\end{equation}
	Substituting this into \eqref{eq:normal-estimate-preliminary} and choosing $\varrho$ sufficiently small yields
	\begin{equation}\label{eq:normal-tangential-estimate-final}
		\begin{aligned}
			&\frac{1}{2\tau}\bigl[A_{\Gamma_{h,*}^{m+1}}^N(e_x^{m+1},e_x^{m+1})-A_{\Gamma_{h,*}^m}^N(e_x^m,e_x^m)\bigr]+\frac\tau2\norm{\nabla_{\Gamma_{h,*}^m}e_V^{m+1}}{L^2(\Gamma_{h,*}^m)}^2 \\
			&\quad + c \left(\norm{e_V^{m+1}}{L^2(\Gamma_{h,*}^m)}^2 + \norm{I_h(\overline T_{h,*}^me_V^{m+1})}{H^1(\Gamma_{h,*}^m)}^2\right) \\
			&\quad\leq C\left((\delta^m)^2+\norm{e_x^m}{H^1(\Gamma_{h,*}^m)}^2+\norm{e_x^{m+1}}{H^1(\Gamma_{h,*}^{m+1})}^2\right).
		\end{aligned}
	\end{equation}

	\emph{Recovery of the full norm.} Let
	\begin{equation*}
		\mathcal E^m =A_{\Gamma_{h,*}^m}^N(e_x^m,e_x^m) + \norm{e_x^m}{L^2(\Gamma_{h,*}^m)}^2 +\norm{I_h(P_*^me_x^m)}{H^1(\Gamma_{h,*}^m)}^2.
	\end{equation*}
	We first show that $\mathcal E^m$ is equivalent to $\norm{e_x^m}{H^1(\Gamma_{h,*}^m)}^2$, i.e.,
	\begin{equation}\label{eq:energy-equivalence}
		C^{-1}\norm{e_x^m}{H^1(\Gamma_{h,*}^m)}^2\le\mathcal E^m\le C\norm{e_x^m}{H^1(\Gamma_{h,*}^m)}^2 .
	\end{equation}
	The upper bound follows from $A_{\Gamma_{h,*}^m}^N(e_x^m,e_x^m)\le\norm{\nabla_{\Gamma_{h,*}^m}e_x^m}{L^2(\Gamma_{h,*}^m)}^2$. 
	
	For the lower bound, we differentiate $e_x^m=I_h(P_*^me_x^m)+(e_x^m\cdot n_*^m)n_*^m+(1-I_h)(P_*^me_x^m)$. Since $\norm{n_*^m-n_{h,*}^m}{L^\infty(\Gamma_{h,*}^m)}\le Ch^{k-1}$ by \eqref{eq:normal-estimates} and $\norm{(1-I_h)(P_*^me_x^m)}{H^1(\Gamma_{h,*}^m)}\le Ch\norm{e_x^m}{H^1(\Gamma_{h,*}^m)}$ by \eqref{eq:P-superapproximation}, the inverse estimate \eqref{eq:inverse} gives
	\begin{equation*}
		\begin{aligned}
			\norm{\nabla_{\Gamma_{h,*}^m}e_x^m}{L^2(\Gamma_{h,*}^m)} &\le\norm{I_h(P_*^me_x^m)}{H^1(\Gamma_{h,*}^m)} +A_{\Gamma_{h,*}^m}^N(e_x^m,e_x^m)^{1/2} \\
			&\quad+C\norm{e_x^m}{L^2(\Gamma_{h,*}^m)} +Ch\norm{e_x^m}{H^1(\Gamma_{h,*}^m)}.
		\end{aligned}
	\end{equation*}
	For sufficiently small $h$, the last term is absorbed, which proves \eqref{eq:energy-equivalence}.

	Since $e_x^{m+1}=e_x^m+\tau e_V^{m+1}$,
	\begin{equation*}
		I_h(P_*^{m+1}e_x^{m+1}) =I_h(P_*^me_x^m)+\tau I_h(P_*^me_V^{m+1})+I_h\bigl((P_*^{m+1}-P_*^m)e_x^{m+1}\bigr),
	\end{equation*}
	where, by $|P_*^m-\overline T_{h,*}^m|\le C|\overline n_{h,*}^m-n_*^m|$, \eqref{eq:normal-estimates}, \eqref{eq:inverse} and Assumption~\ref{ass:regularity},
	\begin{equation*}
		\begin{aligned}
			\norm{I_h(P_*^me_V^{m+1})}{H^1(\Gamma_{h,*}^m)} &\le\norm{I_h(\overline T_{h,*}^me_V^{m+1})}{H^1(\Gamma_{h,*}^m)}+Ch^{k-2}\norm{e_V^{m+1}}{L^2(\Gamma_{h,*}^m)}, \\
			\norm{I_h\bigl((P_*^{m+1}-P_*^m)e_x^{m+1}\bigr)}{H^1(\Gamma_{h,*}^m)}&\le C\tau\norm{e_x^{m+1}}{H^1(\Gamma_{h,*}^{m})}.
		\end{aligned}
	\end{equation*}
	By \eqref{eq:surface-change}, replacing $\Gamma_{h,*}^{m+1}$ by $\Gamma_{h,*}^m$ in the norms of $e_x^{m+1}$ and $I_h(P_*^{m+1}e_x^{m+1})$ costs at most $C\tau\norm{e_x^{m+1}}{H^1(\Gamma_{h,*}^m)}^2$. On $\Gamma_{h,*}^m$ we expand the squares using $e_x^{m+1}=e_x^m+\tau e_V^{m+1}$ and the above decomposition, and apply Young's inequality. For sufficiently small $h$, this gives
	\begin{equation*}
		\begin{aligned}
			&\frac1\tau\Bigl[ \norm{e_x^{m+1}}{L^2(\Gamma_{h,*}^{m+1})}^2+\norm{I_h(P_*^{m+1}e_x^{m+1})}{H^1(\Gamma_{h,*}^{m+1})}^2 \\
			&\qquad -\norm{e_x^m}{L^2(\Gamma_{h,*}^m)}^2-\norm{I_h(P_*^me_x^m)}{H^1(\Gamma_{h,*}^m)}^2\Bigr] \\
			&\le \varepsilon\bigl(\norm{e_V^{m+1}}{L^2(\Gamma_{h,*}^m)}^2 +\norm{I_h(\overline T_{h,*}^me_V^{m+1})}{H^1(\Gamma_{h,*}^m)}^2\bigr) \\
			&\quad+C(\varepsilon)\bigl(\norm{e_x^m}{H^1(\Gamma_{h,*}^m)}^2+\norm{e_x^{m+1}}{H^1(\Gamma_{h,*}^{m})}^2\bigr).
		\end{aligned}
	\end{equation*}
	Adding this to twice \eqref{eq:normal-tangential-estimate-final}, using \eqref{eq:energy-equivalence} and choosing $\varepsilon$ sufficiently small, we obtain
	\begin{equation}\label{eq:total-energy}
		\frac{\mathcal E^{m+1}-\mathcal E^m}{\tau} \le C\bigl(\mathcal E^m+\mathcal E^{m+1}+(\delta^m)^2\bigr).
	\end{equation}
	Since $e_x^0=0$, the discrete Gronwall inequality, \eqref{eq:energy-equivalence} and \eqref{eq:effective-defect} give
	\begin{equation}\label{eq:final-error-bound}
		\max_{0\le j\le m+1}\norm{e_x^j}{H^1(\Gamma_{h,*}^j)}^2 \le C\sum_{j=0}^m\tau(\delta^j)^2\le Ch^8.
	\end{equation}

	Finally, by \eqref{eq:final-error-bound}, $\norm{e_x^{m+1}}{H^1(\Gamma_{h,*}^{m+1})}\le Ch^4\le\frac12h^{7/2}$ for sufficiently small $h$. Hence \eqref{eq:induction} holds for $j=m+1$, which proves \eqref{eq:error-estimate}. The estimate \eqref{eq:error-estimate-Y0} follows from \eqref{eq:error-estimate}, the interpolation error \eqref{eq:interpolation}, Lemma~\ref{lem:flow-comparison} and the norm equivalence.
\end{proof}

\subsection{Proof of the main theorem}\label{subsec:manifold-distance}

\begin{lemma}\label{lem:manifold-distance}
	Let $W:\Gamma^0\to\R^3$ be a regular $C^2$ embedding, and let $\delta$, $a$, $d$ be as in \eqref{eq:signed-distance-representation} for $\Gamma[W]$. If $U:\Gamma^0\to D_\delta(\Gamma[W])$ is a Lipschitz embedding and $a\circ U:\Gamma^0\to\Gamma[W]$ is bijective, then
	\begin{equation}\label{eq:manifold-L1}
		\dman(\Gamma[U],\Gamma[W]) \le C\norm{\nabla_{\Gamma^0}U}{L^\infty(\Gamma^0)}^2\norm{U-W}{L^1(\Gamma^0)}.
	\end{equation}
\end{lemma}
\begin{proof}
	Consider the homotopy
	\begin{equation*}
		H(p,\theta)=\theta U(p)+(1-\theta)a(U(p))=a(U(p))+\theta\,d(U(p))\,n(a(U(p)))
	\end{equation*}
	for $(p,\theta)\in\Gamma^0\times[0,1]$. Since $a\circ U$ is bijective, $\Omega[U]\triangle\Omega[W]\subset H(\Gamma^0\times[0,1])$. Moreover, $\partial_\theta H(p,\theta)=d(U(p))\,n(a(U(p)))$, and $|\nabla_{\Gamma^0}H|\le C|\nabla_{\Gamma^0}U|$ since $a$ is Lipschitz in $D_\delta(\Gamma[W])$ \cite[Section~2.1]{Demlow}. Hence the Jacobian of $H$ is bounded by $C|d(U(p))|\,|\nabla_{\Gamma^0}U(p)|^2$, and the area formula gives
	\begin{equation*}
		\begin{aligned}
			\dman(\Gamma[U],\Gamma[W])&\le C\norm{\nabla_{\Gamma^0}U}{L^\infty(\Gamma^0)}^2\left\|d(U(p))\right\|_{L^1(\Gamma^0)} \\
			&\le C\norm{\nabla_{\Gamma^0}U}{L^\infty(\Gamma^0)}^2\norm{U-W}{L^1(\Gamma^0)},
		\end{aligned}
	\end{equation*}
	where we used $|d(U(p))|\le|U(p)-W(p)|$, since $W(p)\in\Gamma[W]$.
\end{proof}

\begin{proof}[Proof of Theorem~\ref{thm:main}]
	Since $V_0[Y^{[0]}]\cdot n[Y^{[0]}]=-H[Y^{[0]}]$, the flow $Y^{[0]}$ differs from the mean curvature flow only by a tangential velocity, so that $\Gamma[Y^{[0]}(t_m)]=\Gamma[X(t_m)]$. Let $U=(X_h^m)^{\ell_0}$ and $W=Y^{[0]}(t_m)$. By \eqref{eq:error-estimate-Y0},
	\begin{equation}\label{eq:UW-L2}
		\norm{U-W}{L^2(\Gamma^0)}\le C h^2.
	\end{equation}
	For the $W^{1,\infty}$ norm, we split
	\begin{equation*}
		U-W=\bigl(e_x^m\circ X_{h,*}^m\bigr)^{\ell_0}+\bigl((I_hY^{[1]}(t_m))^{\ell_0}-Y^{[1]}(t_m)\bigr)+\bigl(Y^{[1]}(t_m)-Y^{[0]}(t_m)\bigr).
	\end{equation*}
	The first term is bounded by $Ch^{-1}\norm{e_x^m}{H^1(\Gamma_{h,*}^m)}\le Ch^3$ by \eqref{eq:inverse} and \eqref{eq:error-estimate}, and the second by $Ch^k$. For the third term, the interpolation inequality $\norm{\nabla_{\Gamma^0}(Y^{[1]}-Y^{[0]})(t_m)}{L^\infty(\Gamma^0)}\le C\norm{\nabla_{\Gamma^0}(Y^{[1]}-Y^{[0]})(t_m)}{L^2(\Gamma^0)}^{1/2}\norm{\nabla_{\Gamma^0}(Y^{[1]}-Y^{[0]})(t_m)}{H^2(\Gamma^0)}^{1/2}$, \eqref{eq:flow-comparison} and \eqref{eq:regularity} give the bound $C\tau^{1/2}$. Hence
	\begin{equation}\label{eq:UW-W1inf}
		\norm{\nabla_{\Gamma^0}U}{L^\infty(\Gamma^0)} \leq \norm{W}{W^{1,\infty}(\Gamma^0)} + \norm{U-W}{W^{1,\infty}(\Gamma^0)}\le C + Ch \leq C.
	\end{equation}
	Finally, Lemma~\ref{lem:manifold-distance} and \eqref{eq:UW-L2}, \eqref{eq:UW-W1inf} give
	\begin{equation*}
		\dman\bigl(\Gamma_h^m,\Gamma[X(t_m)]\bigr)\le C\norm{U-W}{L^1(\Gamma^0)}\le C(\tau+h^4+h^{k+1})\le Ch^2,
	\end{equation*}
	which is \eqref{eq:main-manifold}.
\end{proof}

\appendix
\section{Proof of Lemma~\ref{lem:tangential-identity}}\label{app:tangential-identity}

\begin{proof}[Proof of Lemma~\ref{lem:tangential-identity}]
	By the assumption $\norm{e_x}{H^1(\Gamma_{h,*})}\le h^{7/2}$ and the inverse estimate \eqref{eq:inverse}, $\norm{\nabla_{\Gamma_{h,*}}e_x}{L^\infty(\Gamma_{h,*})}\le Ch^{5/2}$.

	Following \cite[Section~3.3]{HLT}, we introduce the intermediate surface $\Gamma_{h,\theta}$ between $\Gamma_{h,*}$ and $\Gamma_h$ by
	\begin{equation*}
		\Gamma_{h,\theta}=(1-\theta)\Gamma_{h,*}+\theta\Gamma_h=(\Id+\theta e_x)(\Gamma_{h,*}),\qquad\theta\in[0,1],
	\end{equation*}
	with unit normal $n_{h,\theta}$ satisfying $n_{h,0}=n_{h,*}$ and $n_{h,1}=n_h$. For a function $f_\theta$ on $\Gamma_{h,\theta}$, the material derivative is
	\begin{equation*}
		\partial_\theta^\bullet f_\theta=\frac{\mathrm d}{\mathrm d\theta}\bigl(f_\theta\circ(\Id+\theta e_x)\bigr)\circ(\Id+\theta e_x)^{-1}.
	\end{equation*}
	For $f_\theta$ and $g_\theta$ on $\Gamma_{h,\theta}$, we have \cite[Lemma~5.1]{HLT}
	\begin{equation}\label{eq:transport}
		\frac{\mathrm d}{\mathrm d\theta}\int_{\Gamma_{h,\theta}}f_\theta g_\theta=\int_{\Gamma_{h,\theta}}\partial_\theta^\bullet f_\theta\,g_\theta +\int_{\Gamma_{h,\theta}}f_\theta\,\partial_\theta^\bullet g_\theta +\int_{\Gamma_{h,\theta}}f_\theta g_\theta\,\nabla_{\Gamma_{h,\theta}}\cdot e_x.
	\end{equation}
	Denote the $L^2$ projection onto $S_h(\Gamma_{h,\theta})$ by $\Pi_{h,\theta}$. From \cite[(3.26), (3.27)]{HLT} and \eqref{eq:material-identities}, we have
	\begin{align}
		\partial_\theta^\bullet\Pi_{h,\theta} f_\theta &=\Pi_{h,\theta}\partial_\theta^\bullet f_\theta - \Pi_{h,\theta}\left(\bigl(\Pi_{h,\theta}f_\theta-f_\theta\bigr)\nabla_{\Gamma_{h,\theta}}\cdot e_x\right), \label{eq:projection-derivative} \\
		\partial_\theta^\bullet n_{h,\theta} &=-(\nabla_{\Gamma_{h,\theta}}e_x)n_{h,\theta}. \label{eq:normal-derivative}
	\end{align}
	Since $\norm{\nabla_{\Gamma_{h,*}}e_x}{L^\infty(\Gamma_{h,*})}\le Ch^{5/2}\le\frac12$ for sufficiently small $h$, \cite[Lemma~4.3]{KLLP} yields the norm equivalences between $\Gamma_{h,\theta}$ and $\Gamma_{h,*}$ uniformly in $\theta\in[0,1]$ for every $f_\theta$ with $\partial_\theta^\bullet f_\theta=0$:
	\begin{equation}\label{eq:gradient-equivalence}
		c\norm{f_0}{W^{s,p}(\Gamma_{h,*})}\le\norm{f_\theta}{W^{s,p}(\Gamma_{h,\theta})}\le C\norm{f_0}{W^{s,p}(\Gamma_{h,*})}, \qquad s=0,1,\quad 1\le p\le\infty.
	\end{equation}
	Specifically, for $f_\theta=e_x$, we have
	\begin{equation}\label{eq:gradient-bounds}
		\norm{\nabla_{\Gamma_{h,\theta}}e_x}{L^2(\Gamma_{h,\theta})}\le C\norm{e_x}{H^1(\Gamma_{h,*})}, \quad \norm{\nabla_{\Gamma_{h,\theta}}e_x}{L^\infty(\Gamma_{h,\theta})} \le C\norm{e_x}{W^{1,\infty}(\Gamma_{h,*})}\le Ch^{5/2}.
	\end{equation}

	\emph{The mass term.} By \eqref{eq:J-definition}, \eqref{eq:transport} and \eqref{eq:normal-derivative},
	\begin{equation}\label{eq:J-splitting}
		\begin{aligned}
			J(\phi_h) &=\int_0^1\frac{\mathrm d}{\mathrm d\theta}\int_{\Gamma_{h,\theta}}\Pi_{h,\theta}(\overline V_{h,*}\cdot n_{h,\theta})\,(\phi_h\cdot n_{h,\theta})\dd\theta \\
			&=\int_0^1\int_{\Gamma_{h,\theta}}\partial_\theta^\bullet\Pi_{h,\theta}(\overline V_{h,*}\cdot n_{h,\theta})\,(\phi_h\cdot n_{h,\theta})\dd\theta \\
			&\quad+\int_0^1\int_{\Gamma_{h,\theta}}\Pi_{h,\theta}(\overline V_{h,*}\cdot n_{h,\theta})\,(\phi_h\cdot n_{h,\theta})\,\nabla_{\Gamma_{h,\theta}}\cdot e_x\dd\theta \\
			&\quad-\int_0^1\int_{\Gamma_{h,\theta}}\Pi_{h,\theta}(\overline V_{h,*}\cdot n_{h,\theta})\,[(\nabla_{\Gamma_{h,\theta}}e_x)n_{h,\theta}]\cdot\phi_h\dd\theta \\
			&=:J_1(\phi_h)+J_2(\phi_h)+J_3(\phi_h).
		\end{aligned}
	\end{equation}
	By the $L^2$ and $L^\infty$ stability of $\Pi_{h,\theta}$, \eqref{eq:projection-derivative}, \eqref{eq:normal-derivative} and \eqref{eq:gradient-bounds},
	\begin{equation}\label{eq:projection-bounds}
		\begin{aligned}
			\norm{\Pi_{h,\theta}(\overline V_{h,*}\cdot n_{h,\theta})}{L^\infty(\Gamma_{h,\theta})} &\le C\norm{\overline V_{h,*}}{L^\infty(\Gamma_{h,*})}\le C, \\
			\norm{\partial_\theta^\bullet\Pi_{h,\theta}(\overline V_{h,*}\cdot n_{h,\theta})}{L^2(\Gamma_{h,\theta})} &\le C\norm{\nabla_{\Gamma_{h,\theta}}e_x}{L^2(\Gamma_{h,\theta})} \le C\norm{e_x}{H^1(\Gamma_{h,*})}.
		\end{aligned}
	\end{equation}
	By Lemma~\ref{lem:normal-estimates}, \eqref{eq:normal-derivative}, \eqref{eq:gradient-bounds}, \eqref{eq:gradient-equivalence} and \eqref{eq:lift-norm},
	\begin{equation}\label{eq:normal-components}
		\begin{aligned}
			\norm{n_{h,\theta}-n_*}{L^\infty(\Gamma_{h,*})} &=\Bigl\|n_{h,*}-n_*-\int_0^\theta(\nabla_{\Gamma_{h,s}}e_x)n_{h,s}\dd s\Bigr\|_{L^\infty(\Gamma_{h,*})} \le Ch^{5/2}, \\
			\norm{\phi_h\cdot n_{h,\theta}}{L^2(\Gamma_{h,\theta})} &\le\norm{\phi_h\cdot n_*}{L^2(\Gamma_{h,\theta})} +\norm{n_{h,\theta}-n_*}{L^\infty(\Gamma_{h,*})}\norm{\phi_h}{L^2(\Gamma_{h,\theta})} \\
			&\le C\norm{\phi_h^\ell\cdot n}{L^2(\Gamma)}+Ch^{5/2}\norm{\phi_h}{L^2(\Gamma_{h,*})}.
		\end{aligned}
	\end{equation}
	Hence, by \eqref{eq:projection-bounds}, \eqref{eq:gradient-bounds} and \eqref{eq:normal-components},
	\begin{equation*}
		\begin{aligned}
			&|J_1(\phi_h)|+|J_2(\phi_h)| \\
			&\le\sup_{\theta\in[0,1]}\Bigl(\norm{\partial_\theta^\bullet\Pi_{h,\theta}(\overline V_{h,*}\cdot n_{h,\theta})}{L^2(\Gamma_{h,\theta})} +C\norm{\nabla_{\Gamma_{h,\theta}}\cdot e_x}{L^2(\Gamma_{h,\theta})}\Bigr)\norm{\phi_h\cdot n_{h,\theta}}{L^2(\Gamma_{h,\theta})} \\
			&\le C\norm{e_x}{H^1(\Gamma_{h,*})}\bigl(\norm{\phi_h^\ell\cdot n}{L^2(\Gamma)}+h^{5/2}\norm{\phi_h}{L^2(\Gamma_{h,*})}\bigr).
		\end{aligned}
	\end{equation*}
	For $J_3(\phi_h)$, we further write it as
	\begin{equation}\label{eq:J3-splitting}
		\begin{aligned}
			J_3(\phi_h) &=-\int_0^1\left(\int_{\Gamma_{h,\theta}}\Pi_{h,\theta}(\overline V_{h,*}\cdot n_{h,\theta})\,[(\nabla_{\Gamma_{h,\theta}}e_x)n_{h,\theta}]\cdot\phi_h \right. \\
			&\qquad \left.-\int_{\Gamma_{h,*}}\Pi_{h,*}(\overline V_{h,*}\cdot n_{h,*})\,[(\nabla_{\Gamma_{h,*}}e_x)n_{h,*}]\cdot\phi_h\right)\dd\theta \\
			&\quad-\int_{\Gamma_{h,*}}\bigl[\Pi_{h,*}(\overline V_{h,*}\cdot n_{h,*})-(\overline V_\tau\cdot n)^{-\ell}\bigr]\,[(\nabla_{\Gamma_{h,*}}e_x)n_{h,*}]\cdot\phi_h \\
			&\quad-\int_{\Gamma_{h,*}}(\overline V_\tau\cdot n)^{-\ell}\,[(\nabla_{\Gamma_{h,*}}e_x)(n_{h,*}-n_*)]\cdot\phi_h \\
			&\quad-\Bigl(\int_{\Gamma_{h,*}}(\overline V_\tau\cdot n)^{-\ell}\,[(\nabla_{\Gamma_{h,*}}e_x)n_*]\cdot\phi_h -\int_\Gamma(\overline V_\tau\cdot n)\,[(\nabla_\Gamma e_x^\ell)n]\cdot\phi_h^\ell\Bigr) \\
			&\quad-\int_\Gamma(\overline V_\tau\cdot n)\,[(\nabla_\Gamma e_x^\ell)n]\cdot\phi_h^\ell \\
			&=:J_{31}(\phi_h)+J_{32}(\phi_h)+J_{33}(\phi_h)+J_{34}(\phi_h) -\int_\Gamma(\overline V_\tau\cdot n)\,[(\nabla_\Gamma e_x^\ell)n]\cdot\phi_h^\ell .
		\end{aligned}
	\end{equation}
	For $J_{31}(\phi_h)$, \eqref{eq:commutator} with $\partial_\theta^\bullet e_x=0$ and \eqref{eq:normal-derivative} give
	\begin{equation}\label{eq:gradient-derivative}
		\begin{aligned}
			\partial_\theta^\bullet(\nabla_{\Gamma_{h,\theta}}e_x) &=-(\nabla_{\Gamma_{h,\theta}}e_x)^2 +n_{h,\theta}\otimes n_{h,\theta}\,(\nabla_{\Gamma_{h,\theta}}e_x)^T(\nabla_{\Gamma_{h,\theta}}e_x), \\
			\partial_\theta^\bullet\bigl((\nabla_{\Gamma_{h,\theta}}e_x)n_{h,\theta}\bigr) &=-2(\nabla_{\Gamma_{h,\theta}}e_x)^2n_{h,\theta} +\bigl|(\nabla_{\Gamma_{h,\theta}}e_x)n_{h,\theta}\bigr|^2n_{h,\theta}.
		\end{aligned}
	\end{equation}
	Hence, by \eqref{eq:transport}, \eqref{eq:projection-bounds}, \eqref{eq:gradient-bounds} and \eqref{eq:gradient-equivalence},
	\begin{equation*}
		\begin{aligned}
			&\left|\frac{\mathrm d}{\mathrm d\theta}\int_{\Gamma_{h,\theta}}\Pi_{h,\theta}(\overline V_{h,*}\cdot n_{h,\theta})\,[(\nabla_{\Gamma_{h,\theta}}e_x)n_{h,\theta}]\cdot\phi_h\right| \\
			&\leq \left|\int_{\Gamma_{h,\theta}}\partial_\theta^\bullet\Pi_{h,\theta}(\overline V_{h,*}\cdot n_{h,\theta})\,[(\nabla_{\Gamma_{h,\theta}}e_x)n_{h,\theta}]\cdot\phi_h\right| \\
			&\qquad+\left|\int_{\Gamma_{h,\theta}}\Pi_{h,\theta}(\overline V_{h,*}\cdot n_{h,\theta})\left(\partial_\theta^\bullet\bigl((\nabla_{\Gamma_{h,\theta}}e_x)n_{h,\theta}\bigr) +(\nabla_{\Gamma_{h,\theta}}\cdot e_x)(\nabla_{\Gamma_{h,\theta}}e_x)n_{h,\theta}\right)\cdot\phi_h\right| \\
			&\le C\norm{\partial_\theta^\bullet\Pi_{h,\theta}(\overline V_{h,*}\cdot n_{h,\theta})}{L^2(\Gamma_{h,\theta})} \norm{\nabla_{\Gamma_{h,\theta}}e_x}{L^\infty(\Gamma_{h,\theta})}\norm{\phi_h}{L^2(\Gamma_{h,\theta})} \\
			& \qquad +C\norm{\nabla_{\Gamma_{h,\theta}}e_x}{L^2(\Gamma_{h,\theta})}\norm{\nabla_{\Gamma_{h,\theta}}e_x}{L^\infty(\Gamma_{h,\theta})}\norm{\phi_h}{L^2(\Gamma_{h,\theta})}\\
			&\le Ch^{5/2}\norm{e_x}{H^1(\Gamma_{h,*})}\norm{\phi_h}{L^2(\Gamma_{h,*})},
		\end{aligned}
	\end{equation*}
	and integrating from $0$ to $\theta$ gives $|J_{31}(\phi_h)|\le Ch^{5/2}\norm{e_x}{H^1(\Gamma_{h,*})}\norm{\phi_h}{L^2(\Gamma_{h,*})}$.

	For $J_{32}(\phi_h)$ and $J_{33}(\phi_h)$, we use the $L^2$ geometric errors
	\begin{equation}\label{eq:L2-geometric-errors}
		\norm{\Pi_{h,*}(\overline V_{h,*}\cdot n_{h,*})-(\overline V_\tau\cdot n)^{-\ell}}{L^2(\Gamma_{h,*})}+\norm{n_{h,*}-n_*}{L^2(\Gamma_{h,*})}\le Ch^k,
	\end{equation}
	which follow from the $L^2$ stability and the best approximation property of $\Pi_{h,*}$, \eqref{eq:interpolation} and Lemma~\ref{lem:normal-estimates}. Since $\norm{\nabla_{\Gamma_{h,*}}e_x}{L^\infty(\Gamma_{h,*})}\le Ch^{-1}\norm{e_x}{H^1(\Gamma_{h,*})}$ by \eqref{eq:inverse}, H\"older's inequality gives
	\begin{equation*}
		|J_{32}(\phi_h)|+|J_{33}(\phi_h)|\le Ch^{k-1}\norm{e_x}{H^1(\Gamma_{h,*})}\norm{\phi_h}{L^2(\Gamma_{h,*})}.
	\end{equation*}
	For $J_{34}(\phi_h)$, the perturbation estimate \eqref{eq:lift-mix} and the inverse estimate \eqref{eq:inverse} give
	\begin{equation*}
		|J_{34}(\phi_h)|\le Ch^k\norm{\phi_h}{L^\infty(\Gamma_{h,*})}\norm{e_x}{H^1(\Gamma_{h,*})} \le Ch^{k-1}\norm{e_x}{H^1(\Gamma_{h,*})}\norm{\phi_h}{L^2(\Gamma_{h,*})}.
	\end{equation*}
	Altogether, we have
	\begin{equation}\label{eq:J-leading}
		\begin{aligned}
			&\left|J(\phi_h)+\int_\Gamma(\overline V_\tau\cdot n)\,[(\nabla_\Gamma e_x^\ell)n]\cdot\phi_h^\ell\right| \\
			&\quad\le C\norm{e_x}{H^1(\Gamma_{h,*})}\left(\norm{\phi_h^\ell\cdot n}{L^2(\Gamma)}+(h^{5/2} + h^{k-1})\norm{\phi_h}{L^2(\Gamma_{h,*})}\right).
		\end{aligned}
	\end{equation}

	\emph{The stiffness term.} By definition, we know $B(e_x,\phi_h)=B_\Gamma(e_x^\ell,\phi_h^\ell)$. For $A_{h,*}^N(e_x,\phi_h)$, we have
	\begin{equation}\label{eq:stiffness-lift}
		\begin{aligned}
			A_{h,*}^N(e_x,\phi_h) &=A_\Gamma^N(e_x^\ell,\phi_h^\ell) +\int_{\Gamma_{h,*}}[(\nabla_{\Gamma_{h,*}}e_x)(n_{h,*}-n_*)]\cdot[(\nabla_{\Gamma_{h,*}}\phi_h)n_{h,*}] \\
			&\quad+\int_{\Gamma_{h,*}}[(\nabla_{\Gamma_{h,*}}e_x)n_*]\cdot[(\nabla_{\Gamma_{h,*}}\phi_h)(n_{h,*}-n_*)] \\
			&\quad+\Bigl(\int_{\Gamma_{h,*}}[(\nabla_{\Gamma_{h,*}}e_x)n_*]\cdot[(\nabla_{\Gamma_{h,*}}\phi_h)n_*] -\int_\Gamma[(\nabla_\Gamma e_x^\ell)n]\cdot[(\nabla_\Gamma\phi_h^\ell)n]\Bigr) \\
			&=:A_\Gamma^N(e_x^\ell,\phi_h^\ell) + A_1(\phi_h)+A_2(\phi_h)+A_3(\phi_h).
		\end{aligned}
	\end{equation}
	As for $J_{33}(\phi_h)$, \eqref{eq:L2-geometric-errors} and \eqref{eq:inverse} give
	\begin{equation*}
		|A_1(\phi_h)|+|A_2(\phi_h)|\le Ch^{k-1}\norm{e_x}{H^1(\Gamma_{h,*})}\norm{\phi_h}{H^1(\Gamma_{h,*})}.
	\end{equation*}
	For $A_3(\phi_h)$, the product rule gives $(\nabla_{\Gamma_{h,*}}\phi_h)n_*=\nabla_{\Gamma_{h,*}}(\phi_h\cdot n_*)-(\nabla_{\Gamma_{h,*}}n_*)\phi_h$, and likewise for $e_x$. Using \eqref{eq:lift-mass}, \eqref{eq:lift-mix}, \eqref{eq:lift-stiffness}, together with $\norm{\nabla_{\Gamma_{h,*}}n_*-(\nabla_\Gamma n)^{-\ell}}{L^\infty(\Gamma_{h,*})}\le Ch^k$ in \cite[Lemma~5.2]{Kovacs}, we deduce that
	\begin{equation*}
		|A_3(\phi_h)|\le Ch^{k-1}\norm{e_x}{H^1(\Gamma_{h,*})}\norm{\phi_h}{H^1(\Gamma_{h,*})}.
	\end{equation*}
	Altogether, we have
	\begin{equation}\label{eq:stiffness-difference}
		\left|A_{h,*}^N(e_x,\phi_h) - A_\Gamma^N(e_x^\ell,\phi_h^\ell)\right| \leq C h^{k-1} \norm{e_x}{H^1(\Gamma_{h,*})}\norm{\phi_h}{H^1(\Gamma_{h,*})}.
	\end{equation}

	It remains to rewrite $A_\Gamma^N+B_\Gamma$ on $\Gamma$. By \cite[(5.8)]{BaiLi}, $-A_\Gamma^T(e,\phi)+B_\Gamma(e,\phi) = \int_\Gamma\nabla_\Gamma\Id\cdot((\nabla_\Gamma\cdot e)I - \nabla_\Gamma e - (\nabla_\Gamma e)^T)\nabla_\Gamma\phi$. For smooth $e$ and $\phi$, let $\Gamma_\theta=(\Id+\theta e)(\Gamma)$ and $\partial_\theta^\bullet\phi=0$. Using \cite[Lemma~5.2]{BaiLi}, the identity $A_{\Gamma}(\Id,\phi)=\int_{\Gamma}Hn\cdot\phi$, the transport formula, \eqref{eq:material-identities} with $V$ replaced by $e$, and integration by parts,
	\begin{equation}\label{eq:continuous-identity-derivative}
		\begin{aligned}
			&A_\Gamma^N(e,\phi)+B_\Gamma(e,\phi) \\
			&= A_\Gamma(e,\phi)+\int_\Gamma\nabla_\Gamma\Id\cdot\left((\nabla_\Gamma\cdot e)I - \nabla_\Gamma e - (\nabla_\Gamma e)^T\right)\nabla_\Gamma\phi \\
			& =\frac{\mathrm d}{\mathrm d\theta}A_{\Gamma_\theta}(\Id,\phi)\Big|_{\theta=0} = \frac{\mathrm d}{\mathrm d\theta}\int_{\Gamma_\theta}Hn\cdot\phi\,\Big|_{\theta=0} \\
			&=\int_\Gamma\left(\partial_\theta^\bullet H\,(\phi\cdot n)+H\,\partial_\theta^\bullet n\cdot\phi+H(\phi\cdot n)\,\nabla_\Gamma\cdot e\right) \\
			&=\int_\Gamma\left(-\Delta_\Gamma(e\cdot n)-|A|^2(e\cdot n)+e\cdot\nabla_\Gamma H\right)(\phi\cdot n) -\int_\Gamma H\,[(\nabla_\Gamma e)n]\cdot\phi \\
			&\quad +\int_\Gamma H(\nabla_\Gamma\cdot e)(\phi\cdot n) \\
			&=\int_\Gamma\nabla_\Gamma(e\cdot n)\cdot\nabla_\Gamma(\phi\cdot n)-\int_\Gamma H\,[(\nabla_\Gamma e)n]\cdot\phi \\
			&\quad +\int_\Gamma\left(-|A|^2(e\cdot n)+e\cdot\nabla_\Gamma H+H\nabla_\Gamma\cdot e\right)(\phi\cdot n),
		\end{aligned}
	\end{equation}
	Since both sides of \eqref{eq:continuous-identity-derivative} are bounded bilinear forms on $H^1(\Gamma)^3\times H^1(\Gamma)^3$, the identity extends by density to all $e,\phi\in H^1(\Gamma)^3$.

	Finally, by \eqref{eq:averaged-velocity-Taylor} and $V_0\cdot n=-H$, we have $\norm{\overline V_\tau\cdot n+H}{L^\infty(\Gamma)}\le C\tau$. Combining this with \eqref{eq:stiffness-difference}, \eqref{eq:continuous-identity-derivative} for $e=e_x^\ell$ and $\phi=\phi_h^\ell$, and \eqref{eq:J-leading}, we obtain
	\begin{equation*}
		\begin{aligned}
			&\left|J(\phi_h)+A_{h,*}^N(e_x,\phi_h)+B(e_x,\phi_h)-\int_{\Gamma}\nabla_{\Gamma}(e_x^\ell\cdot n)\cdot\nabla_{\Gamma}(\phi_h^\ell\cdot n)\right| \\
			&\le\left|J(\phi_h)+\int_\Gamma(\overline V_\tau\cdot n)\,[(\nabla_\Gamma e_x^\ell)n]\cdot\phi_h^\ell\right| +\left|A_{h,*}^N(e_x,\phi_h)-A_\Gamma^N(e_x^\ell,\phi_h^\ell)\right| \\
			&\quad+\left|\int_\Gamma(-\overline V_\tau\cdot n-H)\,[(\nabla_\Gamma e_x^\ell)n]\cdot\phi_h^\ell\right| \\
			&\quad +\left|\int_\Gamma\left(-|A|^2(e_x^\ell\cdot n)+e_x^\ell\cdot\nabla_\Gamma H+H\nabla_\Gamma\cdot e_x^\ell\right)(\phi_h^\ell\cdot n)\right| \\
			&\le C\norm{e_x}{H^1(\Gamma_{h,*})}\left(\norm{\phi_h^\ell\cdot n}{L^2(\Gamma)} +(\tau+h^{5/2} + h^{k-1})\norm{\phi_h}{H^1(\Gamma_{h,*})}\right).
		\end{aligned}
	\end{equation*}
	This proves \eqref{eq:T-T}.
\end{proof}

\end{document}